\documentclass{amsart}
\usepackage{tikz-cd,mathtools,enumerate,hyperref,amssymb}

\newtheorem{theorem}{Theorem}[section]
\newtheorem{proposition}[theorem]{Proposition}
\newtheorem{definition}[theorem]{Definition}
\newtheorem{corollary}[theorem]{Corollary}
\newtheorem{lemma}[theorem]{Lemma}

\numberwithin{equation}{section}

\newtheorem{question}[theorem]{Question}

\newcommand{\norm}[1]{\lVert#1\rVert}
\newcommand{\be}{\mathbf{e}}
\newcommand{\bff}{\mathbf{f}}
\newcommand{\bl}{\mathbf{l}}
\newcommand{\bx}{\mathbf{x}}
\newcommand{\by}{\mathbf{y}}
\newcommand{\bz}{\mathbf{z}}
\newcommand{\bone}{\mathbf{1}}
\newcommand{\bsigma}{\boldsymbol{\sigma}}
\newcommand{\cA}{\mathcal{A}}
\newcommand{\cC}{\mathcal{C}}
\newcommand{\cI}{\mathcal{I}}
\newcommand{\cJ}{\mathcal{J}}
\newcommand{\cL}{\mathcal{L}}
\newcommand{\cR}{\mathcal{R}}
\newcommand{\cS}{\mathcal{S}}
\newcommand{\CC}{\mathbb{C}}
\newcommand{\NN}{\mathbb{N}}
\newcommand{\QQ}{\mathbb{Q}}
\newcommand{\RR}{\mathbb{R}}
\newcommand{\ZZ}{\mathbb{Z}}
\newcommand{\round}[1]{\lfloor#1\rceil}

\let\emptyset\varnothing

\title[Transcendence of dendric beta-expansions]{Transcendence of dendric beta-expansions with pure discrete spectrum}
\author{Jihyuk Seo}
\author{Wolfgang Steiner}
\address{Université Paris Cit\'e, CNRS, IRIF, F-75013 Paris, France}
\email{jihyuk.seo@irif.fr, steiner@irif.fr}

\thanks{This work was supported by the ERC grant DynAMiCs (101167561) of the European Research Council, by the bilateral grant SYMDYNAR (ANR-23-CE40-0024 and FWF~I~6750) of the Agence Nationale de la Recherche and the Austrian Science Fund, and by the ANR project IZES (ANR-22-CE40-0011).}

\begin{document}

\maketitle

\begin{abstract}
    We prove the transcendence of numbers that have $\beta$-expansions (in algebraic bases $\beta$) belonging to dendric shifts with pure discrete spectrum, under combinatorial conditions on their directive sequences of substitutions. This includes Arnoux--Rauzy sequences on arbitrary alphabets and Cassaigne--Selmer sequences. The proof uses Rauzy fractals and the subspace theorem.
\end{abstract}

\section{Introduction and main result}
In computer science, it is an important question how fast the binary expansion (or base~$b$ expansion for an integer $b \ge 2$) of a real number can be computed.
The $b$-expansions of rational numbers are eventually periodic and can therefore be computed in linear time.
Hartmanis and Stearns~\cite{Hartmanis65computationalcomplexity} naturally asked whether there exists a real algebraic irrational number whose binary (or base~$b$) expansion can be computed in linear time by a Turing machine.
Cobham~\cite{Cobham68hartmanisstearnsproblem} formulated a weaker version of the Hartmanis--Stearns conjecture, replacing Turing machines by finite automata (or, equivalently, by uniform morphisms).
While the Hartmanis--Stearns conjecture remains open, Cobham's conjecture was proved by Adamczewski, Bugeaud, and Luca~\cite{Adamczewski04complexitenombres}, based on Schlickewei's $\mathfrak{p}$-adic subspace theorem. 
Adamczewski and Bugeaud~\cite{Adamczewski07Dynamics$beta$shifts} generalized this result to algebraic bases $\beta>1$, when $\beta$ is a Pisot or a Salem number.
Another proof of the conjecture uses Mahler's method, which was suggested by Cobham and accomplished by Adamczewski and Faverjon~\cite{Adamczewski17MethodeMahler}. 
More precisely, it is proved in \cite{Adamczewski17MethodeMahler} that, for every algebraic number $\beta$ such that $|\beta|>1$ and every infinite word $w=w_0w_1w_2\ldots$ generated by a uniform morphism, the number $(w)_{\beta} = \sum_{i=0}^\infty w_i \beta^{-i}$ is either in $\QQ(\beta)$ or transcendental.

In combinatorics, a well known theorem by Coven and Hedlund~\cite{Coven73Sequencesminimal} provides another characterization of eventually periodic words by factor complexity: having at most $n$ distinct subwords of length $n$ for every sufficiently large $n$.
Sturmian words, introduced by Morse and Hedlund~\cite{Morse40SymbolicDynamics}, are aperiodic words of the lowest factor complexity $n{+}1$ for every $n \ge 1$.
A~natural question is whether there is a Sturmian word $w$ such that $(w)_{b}$ is algebraic.
The first partial result was obtained by Danilov~\cite{Danilov72classestranscendental}, who proved for the Fibonacci word~$w$ that the real number $(w)_{b}$ is transcendental.
Ferenczi and Mauduit~\cite{Ferenczi97TranscendenceNumbers} extended this result to Sturmian words. 
Furthermore, they proved for any Arnoux--Rauzy word~$w$ on the alphabet~$\{ 0,1,2 \}$ that $(w)_{b}$ is transcendental; see Section~\ref{s:AR} for the definition of Arnoux--Rauzy words.
Developing the $\mathfrak{p}$-adic subspace theorem technique of \cite{Adamczewski07Dynamics$beta$shifts}, Kebis et al.~\cite{Kebis24TranscendenceNumbers} 
extended the result for Sturmian words~$w$ to $(w)_\beta$ with an algebraic base~$\beta$ such that $|\beta|>1$. 
For episturmian words, which include Arnoux--Rauzy words, Kebis~\cite{Kebis23TranscendenceNumbers} obtained the dichotomy that $(w)_{\beta}$ is either transcendental or in $\QQ(\beta)$.
Using similar techniques and a so-called refined Diophantine exponent for sequences, Nguyen~\cite{nguyen2026transcendencemeasuresrefineddiophantine,nguyen2026spectrumrefineddiophantineexponent} gave more precise transcendence results according to Mahler's classification.
Moreover, by \cite{Kebis24TranscendenceNumbers,Kebis26TranscendencePisot}, $k$-bonacci words provide only transcendental numbers $(w)_{\beta}$.
It is natural to ask whether this can be extended to all Arnoux--Rauzy words. 

\begin{question}\label{q:transcendence of Arnoux--Rauzy words}
    Is $(w)_{\beta}$ transcendental for each Arnoux--Rauzy word $w$ on the alphabet $\{ 0,1,\dots,d \}$ and each algebraic base $\beta$ such that $|\beta|>1$?
\end{question}

The $k$-bonacci words are also fixed points of (unimodular) Pisot substitutions. 
Assuming the famous Pisot substitution conjecture (see \cite{Akiyama15PisotSubstitution}), Kebis et al.\ \cite{Kebis26TranscendencePisot} proved that fixed points of Pisot substitutions satisfy the $\beta$-transcendence dichotomy $(w)_{\beta} \in \mathbb{Q}(\beta)$ or $(w)_{\beta} \notin \overline{\mathbb{Q}}$. 
The Pisot substitution conjecture is still open but has been proved for large classes of substitutions by Barge ~\cite{Barge16,Barge18Pisotconjecture} and others; see \cite{Akiyama15PisotSubstitution} for a survey.

Most of the above-mentioned words share the following properties:
\begin{itemize}
    \item 
They are elements of symbolic dynamical systems with pure discrete spectrum, i.e., the shift action is (measurably) conjugate to a translation on a compact abelian group. 
In particular, Sturmian shifts are conjugate to translations on~$\RR/\ZZ$, and the Pisot substitution conjecture implies that fixed points of unimodular Pisot substitutions are conjugate to translations on $\RR^d/\ZZ^d$ for some $d \ge 1$.
(However, there are examples of Arnoux--Rauzy words that do not have this property \cite{CFM08}.)
While Sturmian shifts are codings of translations w.r.t.\ intervals, fixed points of unimodular Pisot substitutions define Rauzy fractals that are fundamental domains of $\RR^d/\ZZ^d$ if the Pisot conjecture holds. 
\item
They are elements of $\cS$-adic shifts~$X_{\bsigma}$, symbolic dynamical systems generated by sequences~$\bsigma$ of substitutions (satisfying particular properties). 
The definition of the fractals and the Pisot substitution conjecture can be extended to the $\cS$-adic case; see \cite{Berthe19Geometrydynamics}.
For sequences $\bsigma$ of unimodular substitutions, pure discrete spectrum can be expressed in terms of a tiling property of a collection of fractals~$\cC_{\bsigma}$.
\item
Sturmian shifts and, more generally, Arnoux--Rauzy shifts are dendric shifts, which are defined in terms of return words and studied in a series of papers by Berthé et al.\ \cite{Berthe15Acyclicconnected,Berthe15finiteindex,Berthe15Maximalbifix}.
Dendric shifts are thus natural generalizations of Sturmian words and have the same factor complexity as Arnoux--Rauzy shifts.
(Another well known class of dendric shifts is given by interval exchange transformations. 
By Avila and Forni~\cite{Avila-Forni07}, they have pure discrete spectrum if and only if they are rotations.)
\end{itemize}

We wonder if one can ensure the transcendence of a number generated by words in an $\cS$-adic shift conjugate to a $d$-dimensional toral translation.

\begin{question}\label{q:S-adic transcendence condition}
    Which conditions on directive sequences~$\bsigma$ ensure that every word $w$ in the $\cS$-adic shift~$X_{\bsigma}$ provides a transcendental number $(w)_{\beta}$ for each algebraic base $\beta$ with $|\beta| > 1$?
\end{question}

In the present paper, we give partial answers to Questions~\ref{q:transcendence of Arnoux--Rauzy words} and~\ref{q:S-adic transcendence condition}.
For precise definitions of the notions that are used in the following theorem, we refer to Section~\ref{sec:defex}. 

\begin{theorem} \label{thm:main}
Let $\bsigma$ be a primitive and recurrent sequence of unimodular substitutions over an alphabet $\cA \subset \overline{\QQ}$, and let $\beta \in \overline{\QQ}$ with $|\beta| > 1$.
If the collection~$\cC_{\bsigma}$ forms a tiling with compact sets, then for each word $w \in X_{\bsigma}$, 
$(w)_{\beta}$~is either transcendental or in $\QQ(\beta,\cA)$.
If moreover $X_{\bsigma}$ is dendric,  then $(w)_{\beta}$ is transcendental.
\end{theorem}

In particular, this answers Question \ref{q:transcendence of Arnoux--Rauzy words} for large classes of Arnoux--Rauzy substitutions (see Corollaries~\ref{cor:Transcendence of Arnoux--Rauzy} and~\ref{cor:Transcendence of bounded partial quotient Arnoux--Rauzy}) and also implies the transcendence for a lot of Cassaigne--Selmer shifts (see Corollary~\ref{cor:Transcendence of Cassaigne--Selmer}).
Sufficient conditions for a collection~$\cC_{\bsigma}$ to form a tiling with compact sets can be found e.g.\ in \cite[Theorem~3.1]{Berthe19Geometrydynamics}.

This paper is organized as follows.
In the next section, we give basic definitions, corollaries, and subdivide Theorem~\ref{thm:main} into two statements: 
Words satisfying the conditions of Theorem~\ref{thm:main} are (strong) echoing (Theorem~\ref{thm:S-adic echoing}), and echoing words satisfy the $\beta$-transcendence dichotomy, with strong echoing allowing only for transcendence (Theorem~\ref{thm:echoing transcendence}). 
Note that we have to modify slightly the echoing property, which was introduced by Kebis et al.\ \cite{Kebis24TranscendenceNumbers,Kebis26TranscendencePisot}, in order to be able to prove our results, and we also simplify its definition. 
In Section~\ref{sec:Transcendence of echoing words}, we show Theorem~\ref{thm:echoing transcendence}, using the $\mathfrak{p}$-adic subspace theorem.
The proof follows essentially that of Kebis et al.\ \cite{Kebis26TranscendencePisot}, with necessary adaptations to our new definition of (strong) echoing words and allowing algebraic digits instead of only rational integer digits. 
In Section~\ref{sec:echoing of S-adic words}, we show the echoing part of Theorem~\ref{thm:S-adic echoing}. 
The proof uses Rauzy fractals, which is a higher dimensional analogue of the techniques for Sturmian shifts in Kebis et al.\ \cite{Kebis24TranscendenceNumbers}.
In Section~\ref{sec:Strong echoing of dendric S-adic words}, we finish the proof of Theorem~\ref{thm:S-adic echoing} by showing the strong echoing part, where dendricity plays a crucial role.

\section{Definitions and examples} \label{sec:defex}
\subsection{$\cS$-adic shifts, Rauzy fractals and tilings}
Let $\cA^*$ and~$\cA^{\NN}$ be the sets of finite and infinite words over a finite alphabet~$\cA$.
Let $\cS$ be the set of \emph{substitutions} over~$\cA$, i.e., maps $\sigma: \cA^* \to \cA^*$ satisfying $\sigma(vw) = \sigma(v) \sigma(w)$ for all $v,w \in \cA^*$, which are naturally extended to~$\cA^{\NN}$ by $\sigma(a_0a_1\cdots) = \sigma(a_0) \sigma(a_1) \cdots$.
A word $v \in \cA^*$ is a \emph{factor} of $w \in \cA^* \cup \cA^{\NN}$ if  $w = uvz$ for some $u \in \cA^*$ and $z \in \cA^* \cup \cA^{\NN}$. 
For $w = w_0w_1\cdots \in \cA^{\NN}$, the factors of $w$ are thus of the form
\[
w_{[m,n)} := w_m w_{m+1}\cdots w_{n-1} \qquad (0 \leq m \leq n),
\]
where $w_{[n,n)}$ is the empty word.

For a sequence of substitutions $\bsigma = (\sigma_{n})_{n\in\NN} \in \cS^{\NN}$, we use the notation
\[
\sigma_{[m,n)} := \sigma_m \circ \sigma_{m+1} \circ \cdots \circ \sigma_{n-1} \qquad (0 \leq m \leq n)
\]
for the composition of substitutions rather than for the symbolic concatenation $\sigma_m\sigma_{m+1}\cdots \sigma_{n-1}$, with $\sigma_{[n,n)}$ being the identity map. 
For convenience, we also write $w_{[x,y)} := w_{\lceil x\rceil} w_{\lceil x\rceil+1}\cdots w_{\lceil y \rceil-1}$ for $x,y \in \mathbb{R}$ with $0 \leq x \leq y$, which coincides with the above notation in case $x$ and $y$ are integers.
The \emph{$\cS$-adic language} of~$\bsigma$ is
\[
\cL_{\bsigma} := \{v \in \cA^* \,:\, \mbox{$v$ is a factor of $\sigma_{[0,n)}(a)$ for some $n \ge 0$, $a \in \cA$}\}.
\]
The \emph{$\cS$-adic shift} of~$\bsigma$ is the dynamical system $(X_{\bsigma}, T)$, where
\[
X_{\bsigma} := \{w \in \cA^{\NN} \,:\, \mbox{$v \in \cL_{\bsigma}$ for all factors $v$ of $w$}\}
\]
and $T: \cA^{\NN} \to \cA^{\NN}$ is the left shift map $w_0w_1\cdots \mapsto w_1w_2\cdots$. 
The \emph{abelianization} of a word $w \in \cA^*$ is the vector $\bl(w) := (|w|_a)_{a\in\cA} \in \NN^{\cA}$, where $|w|_a$ denotes the number of occurrences of the letter~$a$ in~$w$.
The \emph{incidence matrix} of a substitution~$\sigma$ over~$\cA$ is
\[
M_\sigma := \left(|\sigma(b)|_a\right)_{a,b\in\cA}.
\]
A~substitution $\sigma$ is called \emph{unimodular} if $|\!\det M_\sigma| = 1$. 
A~sequence of substitutions $ (\sigma_{n})_{n\in\NN}$ is called \emph{primitive} if, for each $m \ge 0$, there exists $n > m$ such that $|\sigma_{[m,n)}(b)|_a \ge 1$ for all $a,b \in \cA$, i.e., $M_{\sigma_{[m,n)}}$ is a positive matrix.
A~sequence $ (\sigma_{n})_{n\in\NN}$ is called \emph{recurrent} if for each $n \ge 1$ there exists $m \ge 1$ such that $(\sigma_m,\dots,\sigma_{m+n-1}) = (\sigma_0,\dots,\sigma_{n-1})$.

The directive sequence $\bsigma$ has \emph{bounded (weak) partial quotients} if, for each $\sigma \in \cS$, the length of blocks $(\sigma,\dots,\sigma)$ in~$\bsigma$ is bounded.

Let $\bsigma = (\sigma_{n})_{n\in\NN} \in \cS^{\NN}$ be a primitive and recurrent sequence of substitutions.
Then there exists $\bff \in \RR_+^{\cA}$ with 
\begin{equation} \label{e:u}
\bigcap_{n\in\NN} M_{\sigma_{[0,n)}} \RR_+^{\cA} = \RR_+ \bff \quad \mbox{and} \quad \norm{\bff}_1 = 1;
\end{equation}
see \cite[Proposition~1.5.5]{Thuswaldner20}.
We call $\bff$ the \emph{frequency vector} of~$\bsigma$. 
The set of \emph{limit words} of~$\bsigma$ is
\[
W_{\bsigma} = \{w \in X_{\bsigma} \,:\, w \in \sigma_{[0,n)}(\cA^{\NN}) \ \mbox{for all}\ n\ge 0\}
\]
(which is a finite set when $\bsigma$ is primitive).
Let $\be_a$, $a \in \cA$, be the canonical unit vectors of~$\RR^{\cA}$.
For $a \in \cA$, let
\[
\cR_{\bsigma}(a) := \overline{\bigg\{\sum_{i=0}^{k-1} (\be_{w_i}-\bff)  \,:\, w_0w_1\cdots \in W_{\bsigma},\, k \ge 0,\ w_k = a\bigg\}}.
\]
be the (subtiles of) \emph{Rauzy fractals} of~$\bsigma$.
Note that $\cR_{\bsigma}(a) \subset \bone^\perp$, where $\bone^\perp$ is the subspace of $\RR^{\cA}$ orthogonal to the diagonal vector $\bone = (1)_{a\in\cA}$. 
Let
\[
\cC_{\bsigma} := \{\bx + \cR_{\bsigma}(a) \,:\, \bx \in \ZZ^{\cA} \cap \bone^\perp,\, a \in \cA\}.
\]
We say that $\cC_{\bsigma}$ forms a \emph{tiling} of~$\bone^\perp$ if it covers~$\bone^\perp$, i.e., $\bigcup_{\bx\in\ZZ^{\cA}\cap\bone^\perp,a\in\cA} (\bx{+}\cR_{\bsigma}(a)) = \bone^\perp$, and, for all $\bx,\by \in \ZZ^{\cA} \cap \bone^\perp$, $a,b \in \cA$, with $\bx \ne \by$ or $a \ne b$, the intersection $(\bx {+} \cR_{\bsigma}(a)) \cap (\by {+} \cR_{\bsigma}(b))$ has zero measure w.r.t.\ the Lebesgue measure on~$\bone^\perp$.
It has compact sets if $\cR_{\bsigma}(a)$ is compact for all $a \in \cA$. 

\subsection{Dendric shifts}
Let $X$ be a shift space over a finite alphabet~$\cA$.
For a word $u \in \cA^*$, a non-empty word $v \in \cA^*$ is called a \emph{return word} of~$u$ in~$X$ if $vu$ is a factor of some $w \in X$, $vu$ starts with~$u$, and $vu$ has exactly two occurrences of~$u$.
Let $R_X(u)$ denote the set of return words of $u$ in~$X$.
(If $u$ is the empty word, then $R_X(u) = \cA$.) 
If $(X,T)$ is minimal, i.e., $\{T^{n}(w) : n\ge 0\}$ is dense in $X$ for all $w \in X$, then $X$ is called \emph{dendric} if, for each finite word $u \in \cA^*$ that is a factor of some $w \in X$, the set $R_X(u)$ forms a basis of the free group~$F_{\cA}$, i.e., the group generated by the letters $a \in \cA$ and their formal inverses~$a^{-1}$, with concatenation as operation.
(For the equivalence of this definition with other defnitions of dendricity, see \cite{Gheeraert25AlgebraicCharacterization}.) 

\subsection{Examples}
\subsubsection{Arnoux--Rauzy shifts}\label{s:AR}
The \emph{Arnoux--Rauzy substitutions} are defined by
\[
    \alpha_i: i \mapsto i, j \mapsto ij \quad (i \in \cA , \ j \in \cA \setminus \{i\} ).
\]
An \emph{Arnoux--Rauzy shift} $X_{\bsigma}$ is an $\cS$-adic shift of a sequence of Arnoux--Rauzy substitutions $\bsigma = (\alpha_{i_{n}})_{n\in\NN}$ where each letter $i \in \cA$ occurs infinitely often in $(i_{n})_{n\in\NN}$.
Then, a directive sequence of Arnoux--Rauzy shift is primitive.
Since Arnoux--Rauzy shifts are dendric (see \cite[Example 3.2]{Berthe15Acyclicconnected}), the following corollary is straightforward from Theorem~\ref{thm:main}.

\begin{corollary}\label{cor:Transcendence of Arnoux--Rauzy}
    Let $X_{\bsigma}$ be an Arnoux--Rauzy shift generated by recurrent sequence~$\bsigma$ over alphabet $\cA \subset \overline{\QQ}$.
    If $\cC_{\bsigma}$ forms a tiling of $\bone^\perp$ with compact sets,
    then $(w)_{\beta}$ is transcendental for all $w \in X_{\bsigma}$ and $\beta \in \overline{\QQ}$ with $|\beta| > 1$.
\end{corollary}

Note that for almost all Arnoux--Rauzy shifts (with respect to an invariant measure of the shift space $(\{\alpha_i : i\in \cA\}^\infty, T)$ that gives positive measure to each cylinder), the collection $\cC_{\bsigma}$ forms a tiling of~$\bone^\perp$ with compact sets by \cite[Theorem~3.7]{Berthe19Geometrydynamics}.
In particular, this holds for each Arnoux--Rauzy shift over 3 letters with bounded partial quotients by \cite[Theorem~3.8]{Berthe19Geometrydynamics}. 
This allows us to partially generalize the transcendence of Arnoux--Rauzy expansions in integer bases  (\cite[Proposition~4]{Ferenczi97TranscendenceNumbers}) to algebraic bases.

\begin{corollary}\label{cor:Transcendence of bounded partial quotient Arnoux--Rauzy}
    Let $\cA \subset \overline{\QQ}$ with $|\cA| = 3$.
    Let $X_{\bsigma} \subset \cA^{\NN}$ be an Arnoux--Rauzy shift generated by a recurrent sequence $\bsigma$ with bounded partial quotients.
    Then $(w)_{\beta}$ is transcendental for all $w \in X_{\bsigma}$ and $\beta \in \overline{\QQ}$ with $|\beta| > 1$.
\end{corollary}

This immediately implies the following.

\begin{corollary}\label{cor:Transcendence of periodic directive sequence Arnoux--Rauzy}
    Let $\cA \subset \overline{\QQ}$ with $|\cA| = 3$.
    Let $X_{\bsigma} \subset \cA^{\NN}$ be an Arnoux--Rauzy shift generated by a purely periodic sequence $\bsigma$.
    Then $(w)_{\beta}$ is transcendental for all $w \in X_{\bsigma}$ and $\beta \in \overline{\QQ}$ with ~$|\beta| > 1$.
\end{corollary}
In particular, every fixed point of primitive compositions of Arnoux--Rauzy substitutions over 3 letters yields a transcendental number for every base $\beta \in \overline{\QQ}$ with ~$|\beta| > 1$.

\subsubsection{Cassaigne--Selmer shifts}
Let $\cA = \{\mathtt{a},\mathtt{b},\mathtt{c} \} \subset \overline{\QQ}$.
The \emph{Cassaigne--Selmer substitutions} are defined by
\[
    \gamma_1 : \left\{ \begin{aligned}
         \mathtt{a} &\mapsto \mathtt{a}, \\ \mathtt{b} &\mapsto \mathtt{a}\mathtt{c}, \\ \mathtt{c} &\mapsto \mathtt{b},
    \end{aligned} \right. \quad
    \gamma_2 : \left\{ \begin{aligned}
        \mathtt{a} &\mapsto \mathtt{b}, \\ \mathtt{b} &\mapsto \mathtt{a}\mathtt{c}, \\ \mathtt{c} &\mapsto \mathtt{c}.
    \end{aligned} \right.
\]
A \emph{Cassaigne--Selmer shift} is an $\cS$-adic shift of $\bsigma = (\gamma_{i_{n}})_{n\in\NN}$ where the sequence $(i_{n})_{n\in\NN} \in \{1,2\}^{\NN}$ does not end with a sequence in $\{11,22\}^{\NN}$.
Such sequences $\bsigma$ are primitive by \cite[Lemma~1]{Cassaigne17SetSequences} and 
Cassaigne--Selmer shifts are dendric by \cite[Proposition~6]{Cassaigne17SetSequences}, thus one obtains similar results to the Arnoux--Rauzy case.

\begin{corollary}\label{cor:Transcendence of Cassaigne--Selmer}
    Let $\cA \subset \overline{\QQ}$ with $|\cA| = 3$.
    Let $X_{\bsigma} \subset \cA^{\NN}$ be an Cassaigne--Selmer shift generated by a recurrent sequence $\bsigma$.
    If $\cC_{\bsigma}$ forms a tiling of $\bone^\perp$ with compact sets,
    then $(w)_{\beta}$ is transcendental for all $w \in X_{\bsigma}$ and $\beta \in \overline{\QQ}$ with $|\beta| > 1$.
\end{corollary}

Note that for almost every Cassaigne--Selmer shift, the collection $\cC_{\bsigma}$ forms a tiling of $\bone^\perp$ with compact sets by \cite[Theorem~6.2]{Berthe23Multidimensionalcontinued}.

\subsection{Echoing words}
Echoing words were introduced in Kebis et al.\,~\cite{Kebis23TranscendenceNumbers,Kebis24TranscendenceNumbers} with the goal of proving the transcendence of numbers of the form
\[
(w_0w_1\cdots)_{\beta} := \sum_{i=0}^\infty \frac{w_i}{\beta^i},
\]
with algebraic number $\beta \in \overline{\QQ}$ such that $|\beta| > 1$ and $w = w_0w_1\cdots \in \cA^{\NN}$ for some finite digit set $\cA \subset \overline{\QQ}$.
As usual, the field $\overline{\QQ}$~denotes the field of algebraic numbers. 

We slightly modify the definition of Kebis et al.\ to suit our needs.
In particular, we consider only mismatches $w_i \ne w_{i+r_{k}}$ on prefixes of $w$ of finite length~$s_{k}$ (depending on the mismatch density~$\varepsilon$) because it is difficult the control the long-range behaviour of $\cS$-adic shifts, and we do not require a condition on the gaps between consecutive intervals in~$\cI_{k}$ (because we will show that large gaps can be achieved by considering a smaller~$\varepsilon$; see Lemma~\ref{lem:Expanding Gap}).

\begin{definition}[Echoing word] \label{def:echoing}
Let $\cA$ be a finite alphabet.
An infinite word $w = w_0w_1\cdots \in \cA^{\NN}$ is called \emph{echoing} if, for each $\varepsilon>0$, there exist $N \in \NN$, unbounded sequences of natural numbers $(r_{k})_{k \in \NN}$, $(s_{k})_{k \in \NN}$, and sets $\cI_{k} \subset \{0,1,\dots,s_{k}{-}1\}$, with the following properties for all $k \in \NN$:
\begin{enumerate}[\upshape (i)]
\itemsep1ex
\item \label{i:241}
$\{0 \le i < s_{k} \,:\, w_i \ne w_{i+r_{k}}\} \subset \cI_{k}$, 
\item \label{i:242}
 $r_{k} + | \cI_{k} | < \varepsilon\, s_{k}$,
\item \label{i:243}
$\cI_{k}$ is a union of at most $N$ intervals of natural numbers. 
\end{enumerate}
\end{definition}

\begin{definition}[$\beta$-echoing, strong echoing word] \label{def:strongechoing}
An echoing word $w = w_0w_1\cdots \in \cA^{\NN}$ with $\cA \subset \overline{\QQ}$ is called \emph{$\beta$-echoing} for $\beta \in \overline{\QQ} \setminus \{0\}$ if, for each $\varepsilon > 0$, the sequences $(r_{k})_{k \in \NN}$ and $(s_{k})_{k \in \NN}$ in Definition~\ref{def:echoing} can be chosen so that there exist sequences of natural numbers $(l_k)_{k\in\NN}$, $(p_{k})_{k\in\NN}$ and $(q_{k})_{k\in\NN}$ with the following properties:
\begin{enumerate}[\upshape (i)]
\itemsep1ex
\setcounter{enumi}{3}
\item \label{eq:upper bound of p,q}
$0 \le l_k \le p_{k} < q_{k} < s_{k}/2$  for all $k \in \NN$,  $\lim_{k\to\infty} l_k/\log s_k = \infty$, 
\item \label{eq:non vanishing}
$\sum_{i=p_{k}}^{q_{k}} w_i \beta^{-i} \ne \sum_{i=p_{k}}^{q_{k}} w_{i+r_{k}} \beta^{-i}$ for all $k \in \NN$, 
\item \label{eq:expanding gap}
$w_i = w_{i+r_{k}}$ for all integers $i \in [p_{k} {-} l_k, p_{k}) \cup (q_{k}, q_{k} {+} l_k]$, $k \in \NN$.
\end{enumerate}

If $w$ is $\beta$-echoing for all $\beta \in \overline{\QQ}$ with $|\beta| > 1$, then $w$ is called \emph{strong echoing}. 
\end{definition}

Note that it is possible to replace the condition $q_k < s_k/2$ in Definition~\ref{def:strongechoing}~(\ref{eq:upper bound of p,q}) by $q_k < \delta s_k$ for a fixed $\delta < 1$.

\begin{theorem} \label{thm:S-adic echoing}
Let $\bsigma$ be a primitive and recurrent sequence of unimodular substitutions.
If the collection $\cC_{\bsigma}$ forms a tiling of $\bone^\perp$ with compact sets, then each word $w \in X_{\bsigma}$ is echoing.
If moreover $X_{\bsigma}$ is dendric, then each $w \in X_{\bsigma}$ is strong echoing.
\end{theorem}

The following theorem states results of Kebis et al.\ \cite{Kebis23TranscendenceNumbers,Kebis24TranscendenceNumbers,Kebis26TranscendencePisot} with our definition of ($\beta$-)echoing words. 

\begin{theorem}\label{thm:echoing transcendence}
Let $w \in \cA^{\NN}$ be an echoing word with $\cA \subset \overline{\QQ}$.
Let $\beta \in \overline{\QQ}$ with~$|\beta|>1$.
Then $(w)_{\beta}$ is either transcendental or in $\QQ(\beta,\cA)$.
If $w$ is $\beta$-echoing, then $(w)_{\beta}$ is transcendental.
\end{theorem}

We divide this theorem into two propositions, Proposition~\ref{prop:Echoing word non algebraic} and \ref{prop:Strong echoing implies transcendence}.

\section{Transcendence of echoing words}\label{sec:Transcendence of echoing words}
\subsection{$\mathfrak{p}$-adic subspace theorem and other tools}
In this subsection, we recall tools from Diophantine approximation.
Let $K$ be a number field with ring of integers~$\mathcal{O}_{K}$.
A~place $\nu$ of~$K$ is an equivalence class of non-trivial absolute values on~$K$ (i.e., $|x|_{\nu} \neq 1$ for some $x \in K$).
Denote by $\mathcal{M}(K)$ the set of all places of~$K$ and by $\mathcal{M}_{\infty}(K)$ the set of all Archimedean places of~$K$. Each Archimedean place~$\nu$ corresponds to an embedding~$\sigma$ of~$K$ into~$\RR$ or to a pair of complex conjugate embeddings $\sigma, \overline{\sigma}$ of $K$ into~$\mathbb{C}$. We call an Archimedean place $\nu$ \emph{real infinite} if it corresponds to a real embedding $\sigma: K \xhookrightarrow{} \RR$, and \emph{complex infinite} if it corresponds to a complex conjugate pair of embeddings $\sigma, \overline{\sigma}: K \xhookrightarrow{} \mathbb{C}$.

Let $\alpha \in K^{\times}$.
The absolute value associated to a place~$\nu$ is defined by
\[
    |\alpha|_{\nu} := \begin{cases}
        | N(\mathfrak{p}) |^{-\mathrm{ord}_{\mathfrak{p}}(\alpha)/[K:\QQ]} & \mbox{if}\ \nu\ \mbox{is non-Archimedean}, \\
        |\sigma(\alpha)|^{1/[K:\QQ]} & \mbox{if}\ \nu\ \mbox{is real infinite}, \\
        |\sigma(\alpha)|^{2/[K:\QQ]} & \mbox{if}\ \nu\ \mbox{is complex infinite}
    \end{cases}
\]
where a prime ideal~$\mathfrak{p}$ of~$\mathcal{O}_{K}$ corresponds to the non-Archimedean place~$\nu$, $N(\mathfrak{p}) = |\mathcal{O}_{K}/\mathfrak{p}|$ and $\mathrm{ord}_{\mathfrak{p}}(\alpha) = \max\{n \in \ZZ : \alpha \in \mathfrak{p}^{n}\}$.
Then all $\alpha \in K^{\times}$ satisfy the product formula
\begin{equation} \label{e:productformula}
\prod_{\nu\in\mathcal{M}(K)} |\alpha|_{\nu} = 1.
\end{equation}

For a point $\bx = (x_1,\dots,x_{n})$ in $K^{n}$, $n \ge 1$, the \emph{(absolute) Weil height} of $\bx$ is 
    \[
        H(\bx) := \prod_{\nu \in \mathcal{M}(K)} |\bx|_{\nu},
    \]
where $|\bx|_{\nu} := \max\{|x_{1}|_{\nu},\dots,|x_{n}|_{\nu}\}$ for all $\nu \in \mathcal{M}(K)$. 
Slightly abusing notation, we use $|\cdot|_{\nu}$ also for the continuation of $|\cdot|_{\nu}$ to the algebraic closure $\overline{\QQ}/K$.
We also denote the height of a polynomial $f(x) = \sum_{j=0}^{n} a_{j} x^{j} \neq 0 $ by $H(f) := H(a_{0} , \dots , a_{n})$ and the height of an algebraic number $\alpha \in \overline{\QQ}$ by $H(\alpha) := H(1,\alpha)$.

For a finite set of places $S \subset \mathcal{M}(K)$ containing $\mathcal{M}_{\infty}(K)$, we call
$\alpha \in K$ an \emph{$S$-integer} if $|\alpha|_{\nu} \leq 1$ for all $\nu \notin S$, and we denote by $\mathcal{O}_{S}$ the ring of $S$-integers of~$K$. 

In the proof of Proposition~\ref{prop:Echoing word non algebraic}, we use the following special case of the $\mathfrak{p}$-adic subspace theorem, as stated in \cite[Theorem~2.1]{Kebis26TranscendencePisot}.

\begin{proposition}\label{prop:p-adic Subspace Theorem}
    Let $K$ be a number field and $S$ a finite set of places with $\mathcal{M}_\infty(K) \subset S \subset \mathcal{M}(K)$. Let 
    $n \geq 2$, $\varepsilon>0$. Let $L$ be a non-zero linear form in $n$ variables with coefficients in $\overline{\QQ}$.
    Let $\nu_{0} \in S$ be a distinguished place. 
    Let $1 \leq j_{0} \leq n$ be distinguished index such that $x_{j_{0}}$ has non-zero coefficient in~$L$.
    Then, the set of solutions $\bx \in (\mathcal{O}_{S})^{n}$ of the inequality
    \begin{equation}\label{eq:cor:p-adic Subspace Theorem}
|L(\bx)|_{\nu_{0}}\prod_{\substack{(\nu,j) \in S\times\{1,\dots,n\} \\ (\nu,j) \neq (\nu_{0},j_{0})}}|x_{j}|_{\nu}  \leq H(\bx)^{ - \varepsilon}
    \end{equation}
    lies in finitely many proper subspaces of~$K^{n}$.
\end{proposition}

The following lemma on common roots of polynomials is proved in \cite[Proposition~2.3]{Lenstra99Findingsmall}; see also \cite[Proposition~2.1]{Kebis26TranscendencePisot}

\begin{lemma}\label{lem:Common root}
Let $f,g,h \in K[x]$, $n \in \NN$, such that $f(x) = x^{n} g(x) + h(x)$, and let $\beta$ be a root of~$f$ which is not a root of unity. 
If $n - \deg(h) > \frac{\log(\ell H(f))}{\log(H(\beta))}$, where $\ell$ is the number of non-zero terms of~$f$, then $\beta$ is a common root of $g$ and~$h$.
\end{lemma}

\subsection{Dichotomy for echoing words}
In this section, we prove the first part of Theorem~\ref{thm:echoing transcendence}, i.e., that $(w)_\beta$ is in $\QQ(\beta,\cA)$ or transcendental when $w$ is echoing.

We first remark that the definition of echoing in \cite{Kebis24TranscendenceNumbers} contains a condition on the gaps between successive intervals in~$\cI_{k}(w)$.
We now show that this condition is implicitly contained in our Definition~\ref{def:echoing}.
For two non-empty finite sets of integers $I,I'$, we write $I < I'$ if $\max I < \min I'$, and we set $d(I,I') := \min I' - \max I$ if $I < I'$.
For two functions $f, g: \NN \to \RR_{>0}$, we denote $f(n) \ll g(n)$ if there exists a constant $C > 0$ such that for all sufficiently large $n$, $C f(n) \leq g(n)$.

\begin{lemma}\label{lem:Expanding Gap}
    Let $w \in \cA^{\NN}$ be an echoing word with $\cA \subset \overline{\QQ}$.
    For each $\varepsilon > 0$, the constant $N \in \NN$ and the sequences $(r_k), (s_k), (\cI_k)$ of Definition~\ref{def:echoing} can be chosen such that strictly increasing sequences $r_k$, $s_k$ and $\cI_{k} = \bigcup_{i=1}^{N} I_{k,j}$ with non-empty ordered intervals of natural numbers $I_{k,1} < I_{k,2} < \cdots < I_{k,N}$ satisfying $0 \in I_{k,1}$ and
    \begin{equation} \label{e:dIk}
        \lim_{k\to\infty} \frac{d(I_{k,j},I_{k,j+1})}{\log s_{k}} = \infty \mbox{ for all } 1 \leq j < N.
    \end{equation}
\end{lemma}

\begin{proof}
Let $\varepsilon > 0$.
Since $w$ is echoing, there exist $N'$, $r'_k$, $s'_k$ and $\cI'_{k}$ with the following properties for all $k \in \NN$: $\{0 \le i < s'_{k} \,:\, w_i \ne w_{i+r'_{k}}\} \subset \cI'_{k}$,  $r'_{k} + |\cI'_{k}| < \varepsilon s'_{k}/3$, and $\cI'_{k} = \bigcup_{i=1}^{N'_{k}} I'_{k,i}$ with non-empty ordered intervals of natural numbers $I'_{k,1} < I'_{k,2} < \cdots < I'_{k,N'_{k}}$ such that $N'_k \le N'$.
If $0 \not\in I'_{k,1}$, one can add the interval~$\{0\}$ to~$\cI'_{k}$ and increase the index~$j$ of $I'_{k,j}$ for $1 \leq j \leq N'_{k}$. 
Since $s'_{k}$ is unbounded, w.l.o.g.,\ assume that $r'_k+\cI'_k < \varepsilon s'_k/2$ and $s'_{k} > (2N'/\varepsilon)^2$ for all $k \in \NN$.
For $k \in \NN$, $1 \le j < N'_k$, if $d(I'_{k,j},I'_{k,j+1}) < (s'_{k})^{1/2}$, we merge $I'_{k,j}$ and $I'_{k,j+1}$ into a new interval $[\min I'_{k,j}, \max I'_{k,j+1}] \cap \NN$.
In this way, we obtain intervals $I''_{k,j}$, $1 \le j < N''_k \le N'_k$, such that $\cI''_k := \bigcup_{j=1}^{N''_k} \cI''_{k,j} \supset \cI'_k$, with
\[
r'_k + |\cI''_k| < r'_k + |\cI'_k| + (N'_k-1) (s'_{k})^{1/2} < \varepsilon s'_k.
\]
By taking a subsequence of $k \in \NN$ where $N''_k$ is constant and $r'_k, s'_k$ are strictly increasing, we have thus found sequences, $r_k$, $s_k$, $\cI_k = \bigcup_{j=1}^N I_{k,j}$ satisfying Definition~\ref{def:echoing}, with $0 \in I_{k,1} < I_{k,2} < \cdots < I_{k,N}$ and $d(I_{k,j},I_{k,j+1}) \ge s_k^{1/2}$.
Since $s_{k}^{1/2}/\log(s_{k}) \to \infty$ when $k \to \infty$, we have 
\[
    \lim_{k \to \infty} \frac{d(I_{k,j},I_{k,j+1})}{\log(s_{k})} = \infty
\]
for all $1 \leq j < N$.
\end{proof}

Now we can prove the first part of Theorem~\ref{thm:echoing transcendence}.
\begin{proposition}\label{prop:Echoing word non algebraic} 
    Let $w \in \cA^{\NN}$ be an echoing word with finite alphabet $\cA \subset \overline{\QQ}$.
    Let $\beta \in \overline{\QQ}$ with $|\beta|>1$.
    Then $(w)_{\beta}$ is either transcendental or lies in $\QQ(\beta,\cA)$.
\end{proposition}
Since the definition of echoing word is changed, we give a modified version of the proof of Kebis et al.\ \cite{Kebis26TranscendencePisot}.
\begin{proof}\label{proof:Echoing imply non-algebraic}
    Suppose that $(w)_{\beta} \neq 0$ is algebraic.
    We will show that $(w)_{\beta} \in \QQ(\beta,\cA)$.
    Let $S$ be the finite set containing all the Archimedean places of $\QQ(\beta,\cA)$ and all finite places $\nu$ of $\QQ(\beta,\cA)$ corresponding to prime ideals $\mathfrak{p}$ of the ring of integers $\mathcal{O}_{\QQ(\beta,\cA)}$ such that $|\beta|_{\nu} \neq 1$ or $|a|_{\nu} \neq 1$ for some $a \in \cA$.
    Let $\nu_{0} \in S$ be the Archimedean place corresponding to the identity embedding $\QQ(\beta,\cA) \hookrightarrow \mathbb{C}: \beta \mapsto \beta$, $a \mapsto a$ for all $a \in \cA$.
    We have $|\beta|_{\nu_0} = |\beta|^{1/[\QQ(\beta,\cA):\QQ]}$ if $\nu_0$ is real infinite and $|\beta|_{\nu_0} = |\beta|^{2/[\QQ(\beta,\cA):\QQ]}$ if $\nu_0$ is complex infinite, thus $|\beta|_{\nu_0} > 1$.
    Let $\kappa = \max_{\nu \in S, a,b \in \cA}(|\beta|_{\nu},|a|_{\nu}, |a-b|_{\nu})$.
    Let $\varepsilon = \frac{\log(|\beta|_{\nu_0})}{4|S| \log(\kappa)}$.
    Let $r_k$, $s_k$, $\cI_{k} = \bigcup_{j=1}^N I_{k,j}$ be as in Lemma~\ref{lem:Expanding Gap}, and set $t_k := \max I_{k,1}$.
    For each $k \in \NN$, we define $\bz_{k} = (z_{k,1},\dots,z_{k,N+2}) \in (\mathcal{O}_{S})^{N+2}$ by
    \begin{equation}
    \begin{split}
        z_{k,1} &:= \beta^{r_{k}}, \quad 
        z_{k,2} := 1, \\
        z_{k,3} &:= \sum_{i=0}^{r_{k}-1} w_{i} \beta^{r_{k}-i} +  \sum_{i \in I_{k,1}} (w_{r_{k}+i} - w_{i})\beta^{-i}, \\
        z_{k,j+2} &:= \sum_{i \in I_{k,j}} (w_{r_{k}+i} - w_{i})\beta^{-i} \quad (j = 2,\dots,N).
        \end{split}
    \end{equation}
    By taking a subsequence of $k \in \NN$, we can assume w.l.o.g.\ that there exists $J \subset \{2,\dots,N\}$ such that $z_{k,j+2} \neq 0$ if and only if $j \in J$. 
    Now, define the linear form $L$ over $\bx = (x_{1}, \dots, x_{|J|+3})$ by
    \begin{equation}\label{eq:Linear Combination}
        L(\bx) := (w)_{\beta} \, x_{1} - (w)_{\beta} \, x_{2} -x_{3}-\sum_{j \in J}  x_{j+2}.
    \end{equation}
    Then
    \begin{equation*}
    \begin{split} 
        L(\bz_{k})
        & = \beta^{r_{k}} (w)_{\beta} - (w)_{\beta} - \sum_{i=0}^{r_{k}-1} w_{i} \beta^{r_{k}-i} -  \sum_{i \in \cI_{k}} (w_{r_{k}+i} - w_{i})\beta^{-i} \\
        &= \sum_{i=s_{k}}^{\infty} (w_{r_{k}+i} - w_{i})\beta^{-i}
    \end{split} 
    \end{equation*}
    and thus
    \begin{equation} \label{eq:Archimedean upper bound}
    |L(\bz_{k})|_{\nu_0} \ll |\beta|_{\nu_0}^{-s_{k}}.
    \end{equation}    
    For the other places $\nu \in S \setminus \{ \nu_{0} \}$, using the product formula $\prod_{\nu \in S} |\beta|_{\nu} = 1$, we obtain that
    \begin{equation}\label{eq:a1 absolute value v0}
        \prod_{\nu \in S \setminus \{ \nu_{0} \}} | z_{k,1} |_{\nu} = |\beta|_{\nu_{0}}^{-r_{k}}.
    \end{equation}
    Recall that $\varepsilon s_{k} > r_{k}+t_{k}$, with $\varepsilon = \frac{\log(|\beta|_{\nu_0})}{4|S| \log(\kappa)}$.
    By the product formula, we also have
    \begin{equation}\label{eq:a3 other absolute value}
        \prod_{\nu \in S} |z_{k,3}|_{\nu} = \prod_{\nu \in S} |\beta^{t_k} z_{k,3}|_{\nu}
        \leq \prod_{\nu \in S} \sum_{i=0}^{r_{k}+t_{k}} \kappa^{i+1}
        \ll \prod_{\nu \in S} \kappa^{r_{k}+t_{k}}
        = \kappa^{(r_{k}+t_{k})|S|}
        \leq |\beta|_{\nu_0}^{s_{k}/4}.
    \end{equation}
    Similarly, using that $|\cI_{k}| < \varepsilon s_{k}$ for all $k \in \NN$, we obtain that
    \begin{equation}
        \label{eq:other absolute value}
        \prod_{j \in J} \prod_{\nu \in S} |z_{k,j+2}|_{\nu}
        \leq \prod_{j \in J} \prod_{\nu \in S} \sum_{i =0}^{|I_{k,j}|-1} \kappa^{i+1} \ll |\beta|_{\nu_0}^{s_{k}/4}.
    \end{equation}
    
    In order to apply Proposition~\ref{prop:p-adic Subspace Theorem}, we use the above inequalities \eqref{eq:Archimedean upper bound}--\eqref{eq:other absolute value} and the trivial equation $\prod_{\nu \in S} |z_{k,2}|_{\nu} = 1$ to obtain
    \begin{equation}\label{eq:asymtotic subspace theorem}
        |L(\bz_{k})|_{\nu_{0}} \prod_{\substack{(\nu,j) \in S\times \{1,2,3\} \\ (\nu,j) \neq (\nu_{0} , 1)}}|z_{k,j}|_{\nu} \prod_{(\nu,j) \in S\times J}|z_{k,j+2}|_{\nu}  
        \ll |\beta|_{\nu_{0}}^{-s_{k}} |\beta|_{\nu_{0}}^{-r_{k}} |\beta|_{\nu_{0}}^{s_{k}/2}.
    \end{equation} 
    Since $r_{k} \to \infty$ as $k \to \infty$, we have 
    \begin{equation}\label{eq:subspace theorem}
        |L(\bz_{k})|_{\nu_{0}} \prod_{\substack{(\nu,j) \in S\times \{1,2,3\} \\ (\nu,j) \neq (\nu_{0} , 1)}}|z_{k,j}|_{\nu} \prod_{(\nu,j) \in S\times J}|z_{k,j+2}|_{\nu}   
        <|\beta|_{\nu_{0}}^{-s_{k}/2}
    \end{equation}
    for sufficiently large~$k$.
    As we calculated the above inequalities \eqref{eq:a1 absolute value v0}--\eqref{eq:other absolute value}, $|z_{k,2}|_{\nu_{0}} = 1$ and as $\bz_{k} \in \mathcal{O}_{S}$, there exists constants $c_{0}, c_{1}>0$ such that the height of $\bz_{k}$ satisfies
    \begin{equation}\label{eq:height of a_n}
        H(\bz_{k}) = \prod_{\nu \in \mathcal{M}(K)} |\bz_{k}|_{\nu} \leq \prod_{\nu \in S} |\bz_{k}|_{\nu} 
        \ll \kappa^{c_{0}s_{k}|S|} 
        < |\beta|_{\nu_{0}}^{c_{1}s_{k}}
    \end{equation}
    Thus the left-hand side of inequality \eqref{eq:subspace theorem} is at most $H(\bz_{k})^{-1/2c_{1}}$ for sufficiently large~$k$.

    By Proposition~\ref{prop:p-adic Subspace Theorem}, we obtain that infinitely many $\bz_{k}$ lie in finitely many proper subspaces of $\QQ(\beta,\cA)^{|J|+3}$.
    Therefore, we can find a proper subspace defined by a linear form 
    \[
        F(\bx) := a_{1}x_{1}+ a_{2}x_{2}+ a_{3}x_{3}+ \sum_{j \in J} a_{j+2}x_{j+2}
    \]
with coefficients $a_i \in \QQ(\beta,\cA)$ such that $F(\bz_{k}) = 0$ for infinitely many~$k$.
    From now on, we consider such subsequence of $k \in \NN$ such that $F(\bz_{k}) = 0$. Then, we have
    \begin{equation}
        \begin{split}\label{eq:Linear Combination 0}
        0 = F(\bz_{k}) &= a_{1}\beta^{r_{k}} + a_{2} + a_{3}\, \bigg( \sum_{i=0}^{r_{k}-1} w_{i} \beta^{r_{k}-i} +  \sum_{i \in I_{k,1}} (w_{r_{k}+i} - w_{i})\beta^{-i} \bigg) \\
        & \qquad + \sum_{j \in J} a_{j+2}\, \bigg( \sum_{i \in I_{k,j}} (w_{r_{k}+i} - w_{i})\beta^{-i} \bigg).
        \end{split}
    \end{equation}
    Multiplying this equation by $\beta^{s_k}$ gives a polynomial $f_{k}(\beta) = 0$ with coefficients in $\QQ(\beta,\cA)$.
    By the form of $f_{k}$, there exists a constant $c_{2} = H(a_1, \dots, a_{|J|+3})$ such that $H(f_{k}) \leq c_{2}\kappa$.
    Since the $I_{k,j}$ are chosen as in Lemma~\ref{lem:Expanding Gap}, we have $d(I_{k,j},I_{k,j+1}) > \frac{\log((r_{k}+s_{k})c_{2}\kappa)}{\log(H(\beta))}$ for all $j < N$ and sufficiently large $k \in\NN$. 
    By Lemma~\ref{lem:Common root}, each monomial of $f_{k}$ have same root $\beta$, thus we obtain that
    \[
    a_{j+2} z_{k,j+2} = 0
    \]
    for all $j \in J$ and sufficiently large $k$.
    Since we have chosen $J$ such that $z_{k,j+2} \neq 0$ if and only if $j \in J$, we obtain that $a_{j+2} = 0$ for all $j \in J$.
    Dividing \eqref{eq:Linear Combination 0} by~$\beta^{r_k}$ and letting $k\to\infty$, we get that
    \begin{equation}\label{eq:}
        a_1 + a_3\, (w)_\beta = 0.
    \end{equation}
     If $a_{3} = 0$, then $a_1 = 0$, thus $a_2 = 0$, contradicting that $F$ defines a proper subspace.
    If $a_3 \ne 0$, then $(w)_\beta \in \QQ(\beta, \cA)$.
\end{proof}

\subsection{Transcendence of $(w)_\beta$ when $w$ is $\beta$-echoing}
Now, prove the second part of Theorem~\ref{thm:echoing transcendence}:

\begin{proposition}\label{prop:Strong echoing implies transcendence}
    Let $\beta \in \overline{\QQ}$ with $|\beta|>1$.
    Let $w \in \cA^{\NN}$ be a $\beta$-echoing word with finite alphabet $\cA \subset \overline{\QQ}$.
    Then $(w)_{\beta}$ is transcendental.
\end{proposition}

We use the following lemma, which follows from \cite[Proposition~2.4.1]{Bilu26Valuesgeneralized} and $|1/P(\alpha)|_{\nu} \leq H(P(\alpha))$.

\begin{lemma}\label{lem:generalized liouville inequality}
    Let $P(x) \in \overline{\QQ}[x]$ be a non-zero polynomial of degree~$n$, let $\alpha \in \overline{\QQ}$ be an algebraic number such that $P(\alpha) \neq 0$, and let $\nu$ be a place of~$\overline{\QQ}$.
    Then, we have
    \[ |P(\alpha)|^{-1}_{\nu} \leq (n {+} 1)\, H(P)\, H(\alpha)^{n}.  \]
\end{lemma}

The following lemma is similar to Lemma~\ref{lem:Expanding Gap}.

\begin{lemma}\label{lem:beta-Expanding Gap}
    Let $\beta \in \overline{\QQ} \setminus \{0\}$ be not a root of unity.
    Let $w \in \cA^{\NN}$ be a $\beta$-echoing word with $\cA \subset \overline{\QQ}$. 
    For each $\varepsilon > 0$, the constant $N \in \NN$ and the sequences  $(l_k), (p_k), (q_k), (r_k), (s_k), (\cI_k)$ of Definitions~\ref{def:echoing} and~\ref{def:strongechoing} can be chosen such that, for all $k \in \NN$, $\cI_{k} = \bigcup_{i=1}^{N} I_{k,j}$ with non-empty ordered intervals of natural numbers $I_{k,1} < I_{k,2} < \cdots < I_{k,N}$ satisfying \eqref{e:dIk} and $p_k, q_k \in I_{k,j_0}$ for some constant $j_0 \in \{2,\dots,N\}$.
\end{lemma}

\begin{proof}
Let $\varepsilon > 0$ and choose $N'$, $r_k$, $s_k$, $\cI'_k = \bigcup_{j=1}^{N'} I'_{k,j}$, $p'_{k}$, $q'_{k}$ and $l'_k$ satisfying the conditions (\ref{i:241})--(\ref{eq:expanding gap}) of Definitions~\ref{def:echoing} and~\ref{def:strongechoing}.
Obtaning $I_{k,j}$ with little modification to $I'_{k,j}$, we show that for every $j \leq N$, either $I_{k,j} \cap  [p'_{k}, q'_{k}] = \emptyset$ or $I_{k,j} \subset [p'_{k}, q'_{k}]$.

Let $\cI_{k} = \cI'_{k} \, \setminus \, ([p'_{k}-l'_{k}, p'_{k}) \cup (q'_{k},q'_{k}+l'_{k}])$.
By the condition~(\ref{eq:expanding gap}) of $\beta$-echoing, $\{0 \le i < s_{k} \,:\, w_i \ne w_{i+r_{k}}\} \subset \cI_{k}$ still holds.

Thus, by taking subsequence of $k \in \NN$, w.l.o.g, we can assume that there exists a number of intervals $N' \leq N \leq N'+2$ such that for all $k \in \NN$ and $1 \leq j \leq N$, $I_{k,j} \cap  [p'_{k}, q'_{k}] = \emptyset$ or $I_{k,j} \subset [p'_{k}, q'_{k}]$ and $0 \in I_{k,1}$ while satisfying Definitions~\ref{def:echoing} and~\ref{def:strongechoing}.

Finally, we claim that for every $k$, we have $\sum_{i \in I_{k,j'_{0}}} (w_{r_{k}+i} {-} w_{i})\beta^{-i} \neq 0$ and $I_{k,j'_0} \subset [p'_{k}, q'_{k}]$ for some $2 \leq j'_0 \leq N$.

Otherwise, for any fixed $k$, let $2 \leq j'_{1} \leq N$ be the largest index $j$ such that $\min I_{k,j} \leq s_k/2$.
Assume that the claim is not true, i.e., 
\[
    \sum_{i \in I_{k,j}} (w_{r_{k}+i} - w_{i})\beta^{-i} = 0 \mbox{ for all } j \mbox{ such that } I_{k,j} \subset [p'_{k}, q'_{k}].
\]
We have 
\[
    \sum_{p'_{k} \leq i \leq q'_{k}} (w_{r_{k}+i} - w_{i})\beta^{-i} = 0
\] for all $k \in \NN$.
Then, by the conditions~\eqref{eq:upper bound of p,q} and \eqref{eq:expanding gap} of $\beta$-echoing and Lemma~\ref{lem:Common root}, this contradicts the condition~(\ref{eq:non vanishing}) of $\beta$-echoing: $\sum_{p'_{k}\leq i \leq q'_{k}} (w_{r_{k}+i} - w_{i})\beta^{-i} \neq 0$.

Again by taking subsequence of $k \in \NN$, W.l.o.g., we can assume that $j_{0}$ is a constant for all $k \in \NN$. 
Then, we can choose as following: $p_{k} = \min I_{k,j_{0}}$, $q_{k} = \max I_{k,j_{0}}$ and $l_{k} = \min(l'_{k}, d(I_{k,j_{0}-1}, I_{k,j_{0}}), d(I_{k,j_{0}}, I_{k,j_{0}+1}))$.
\end{proof}

Since the definition of strong echoing word has been changed and we allow algebraic digits, we give a modified version of the proof of Kebis et al.\ \cite{Kebis26TranscendencePisot}.

\begin{proof}[Proof of Proposition~\ref{prop:Strong echoing implies transcendence}]
    We assume that $(w)_{\beta}$ is algebraic, therefore lies in $\QQ(\beta,\cA)$ by Proposition~\ref{prop:Echoing word non algebraic}, then derive a contradiction.
    We can also assume w.l.o.g.\ that $(w)_{\beta} \ne 0$ because, by \eqref{eq:non vanishing}, we have $(T^{n}(w))_\beta \ne 0$ for some $n \in [p_0,q_0]$, and finite shifts preserve the $\beta$-echoing property and algebraicity.
    At the end of the proof of Proposition~\ref{prop:Echoing word non algebraic}, we have shown that $a_{3} \neq 0$.
    Our goal is to kill the $x_3$ term in the linear form $L$ and to use Proposition~\ref{prop:p-adic Subspace Theorem} again.
    Then Lemmas~\ref{lem:Common root} and~\ref{lem:Expanding Gap} lead us to the contradiction. 
    
    Let $S$, $\nu_{0}$ and $\kappa$ be as in the proof of Proposition~\ref{prop:Echoing word non algebraic}, let~$c_3= \frac{\log(H(\beta))}{\log(|\beta|_{\nu_{0}})} \geq~1$ and $\varepsilon < \min\Big(\frac{\log(|\beta|_{\nu_0})}{4|S| \log(\kappa)},\frac{1}{16 \, c_{3}}\Big) \leq \frac{1}{16}$.
    Let the constants $j_0,N$ and the sequences $r_{k}$, $s_{k}$, $\cI_{k}$, $I_{k,j}$ be as in Lemma~\ref{lem:beta-Expanding Gap}, and let $J$ and $t_{k}=\max I_{k,1}$ be as in the proof of Proposition~\ref{prop:Echoing word non algebraic}.
    From the the proof of Proposition~\ref{prop:Echoing word non algebraic}, we have the non-zero linear form
    \[
        F(\bx) = a_{1}x_{1}+ a_{2}x_{2}+ a_{3}x_{3},
    \]
for which all $\bz_{k} \in \QQ(\beta,\cA)^{|J|+3}$ are solutions of $F(\bz_{k}) = 0$.
    We consider the linear form 
    \[
     L'(\bx) := F(\bx) + a_{3} L(\bx) = (a_{2} - a_{3}(w)_{\beta})x_{2} - \sum_{j \in J} a_{3} x_{j+2}
    \]
    (where no $x_1$ and $x_{3}$ terms occur).

    By Lemma~\ref{lem:beta-Expanding Gap}, there exists $j_0 \in J$ such that $z_{k,j_{0}+2} = \sum_{i \in I_{k,j_{0}}} (w_{r_{k}+i} - w_{i})\beta^{-i} \neq 0$ and $M_k := \max I_{k,j_0} <(\varepsilon+1/2)\, s_k < (9/16)\, s_k$ for sufficiently large~$k$.

    Recall that $2\varepsilon c_{3} < 1/8$.
    For sufficiently large $k \in \NN$, we have $\varepsilon s_{k} < |\beta|_{\nu_{0}}^{\varepsilon s_{k}}$.
    We regard $z_{k,j_{0}+2} \, \beta^{M_{k}}$ as a non-zero polynomial $f(x) \in \QQ(\beta,\cA)[x]$ of degree $|I_{k,j_{0}}| < \varepsilon s_{k}$ evaluated at~$\beta$.
    For the height of $f$, we have $H(f) \leq \kappa$.
    Applying Lemma~\ref{lem:generalized liouville inequality} to the polynomial $f(\beta)$, we have
    \begin{equation}\label{eq:Liouville inequality for non vanishing part}
        |z_{k,j_{0}+2}|^{-1}_{\nu_{0}} \ll (\varepsilon s_{k}) \, \kappa \, H(\beta)^{\varepsilon s_{k}} \, |\beta|^{9s_{k}/16 }_{\nu_{0}} \ll |\beta|_{\nu_{0}}^{(9/16 + 2 \varepsilon c_{3}) \, s_{k}} < |\beta|_{\nu_{0}}^{ 11s_{k}/16 }.
    \end{equation}
       
    We have 
    \begin{equation}\label{eq:Archimedean upper bound of modified linear form}
        |L'(\bz_{k})|_{\nu_{0}} = |a_{3} L(\bz_{k})|_{\nu_{0}} \ll |\beta|_{\nu_{0}}^{-s_{k}}
    \end{equation}
    by the inequality \eqref{eq:Archimedean upper bound}.
    Combining  inequalities \eqref{eq:other absolute value}, \eqref{eq:Liouville inequality for non vanishing part}, \eqref{eq:Archimedean upper bound of modified linear form}, and the trivial equation~$\prod_{\nu \in S} |z_{k,2}|_{\nu} =~1$, we obtain that
    \begin{equation}\label{eq:asymtotic subspace theorem modified}
        |L'(\bz_{k})|_{\nu_{0}} \prod_{\nu \in S}|z_{k,2}|_{\nu} \prod_{\substack{(\nu,j) \in S\times J \\ (\nu,j) \neq (\nu_{0} , j_{0})}}|z_{k,j+2}|_{\nu}
        \ll |\beta|_{\nu_0}^{s_{k}/4} |\beta|_{\nu_{0}}^{11 s_{k}/16} |\beta|_{\nu_{0}}^{-s_{k}}.
    \end{equation}
        Therefore, we have
    \begin{equation}\label{eq:subspace theorem modified}
        |L'(\bz_{k})|_{\nu_{0}} \prod_{\nu \in S}|z_{k,2}|_{\nu} \prod_{\substack{(\nu,j) \in S\times J \\ (\nu,j) \neq (\nu_{0} , j_{0})}}|z_{k,j+2}|_{\nu} 
        <|\beta|_{\nu_{0}}^{-s_{k}/17}
    \end{equation}
    for sufficiently large $k$.
    Recall that the height of $\bz_{k}$ satisfies \eqref{eq:height of a_n}.
    Thus the left-hand side of inequality \eqref{eq:subspace theorem modified} is at most $H(\bz_{k})^{-1/17c_{3}}$ for sufficiently large $k$.

    By Proposition~\ref{prop:p-adic Subspace Theorem}, we know that the set of solutions $\bz_{k}$ to the inequality \eqref{eq:subspace theorem modified} lies in finitely many proper subspaces of $\QQ(\beta,\cA)^{|J|+1}$.
    Let $G$ be a linear form defining one of these proper subspaces which contains infinitely many $\bz_{k}$, and write
    \[
        G(\bx) = b_{2}x_{2}+ \sum_{j \in J} b_{j+2}x_{j+2}.
    \]
    In the same manner as in the proof of Proposition~\ref{prop:Echoing word non algebraic} above, by Lemma~\ref{lem:Common root} applied to the polynomial equation $G(\bz_{k}) = 0$ on~$\beta$, we obtain $b_{j+2}=0$ for all $j \in J$, thus $b_2 = 0$, which contradicts that $G$ forms a proper subspace.
\end{proof}    

\section{Echoing of $\cS$-adic words}\label{sec:echoing of S-adic words}

In this section, we prove the first part of Theorem~\ref{thm:S-adic echoing}.
We start by proving the following lemma; important parts of the proof are illustrated in Figure~\ref{f:R}.
For a word $w = w_0w_1\cdots \in X_{\bsigma}$, where $\bsigma$ is a sequence of substitutions over $\cA$ that admits a frequency vector~$\bff$ as in \eqref{e:u}, we set
\begin{equation} \label{e:Jr}
\cJ_r(w) := \{i \ge 0 \,:\, w_i \ne w_{i+r}\ \mbox{or}\ \bl(w_{[i,i+r)}) \ne \round{r\bff}\}
\end{equation}
for $r \ge 1$, where $\round{\bx}$ denotes the nearest integer vector to a vector~$\bx$.

\begin{lemma} \label{lem:density0}
Let $\bsigma$ be a primitive sequence of unimodular substitutions over~$\cA$ that admits a frequency vector $\bff = (f_a)_{a\in\cA}$, and assume that $\cC_{\bsigma}$ forms a tiling of $\bone^\perp$ with compact sets.
Then for every $\varepsilon > 0$, there exist $r,s \ge 1$ such that the following properties hold for all $w = w_0w_1\cdots \in X_{\bsigma}$, $a \in \cA$.
\begin{enumerate}[\upshape (i)]
\itemsep1ex
\item \label{i:d01} 
$|\{i \in \cJ_r(w) \cap [0,s) : w_i = a\}| \le \varepsilon\, f_a\, s$,
\item \label{i:d02}
$r \le \varepsilon\, f_a \, s$,
\item \label{i:d03}
$|w_{[1,s/2)}|_a \ge f_a\, s/3$, $f_a\, s/9 \le |w_{[1,s/8)}|_a \le f_a\, s/7$, 
\item \label{i:d04}
for each $u \in \cL_{\bsigma} \cap \cA^{r+1}$ and $v \in R_{X_{\bsigma}}(u)$, $vu$ is a factor of~$w_{[1,s/8)}$.
\end{enumerate}
\end{lemma}

\begin{figure}[htbp]
    \centering
    \input{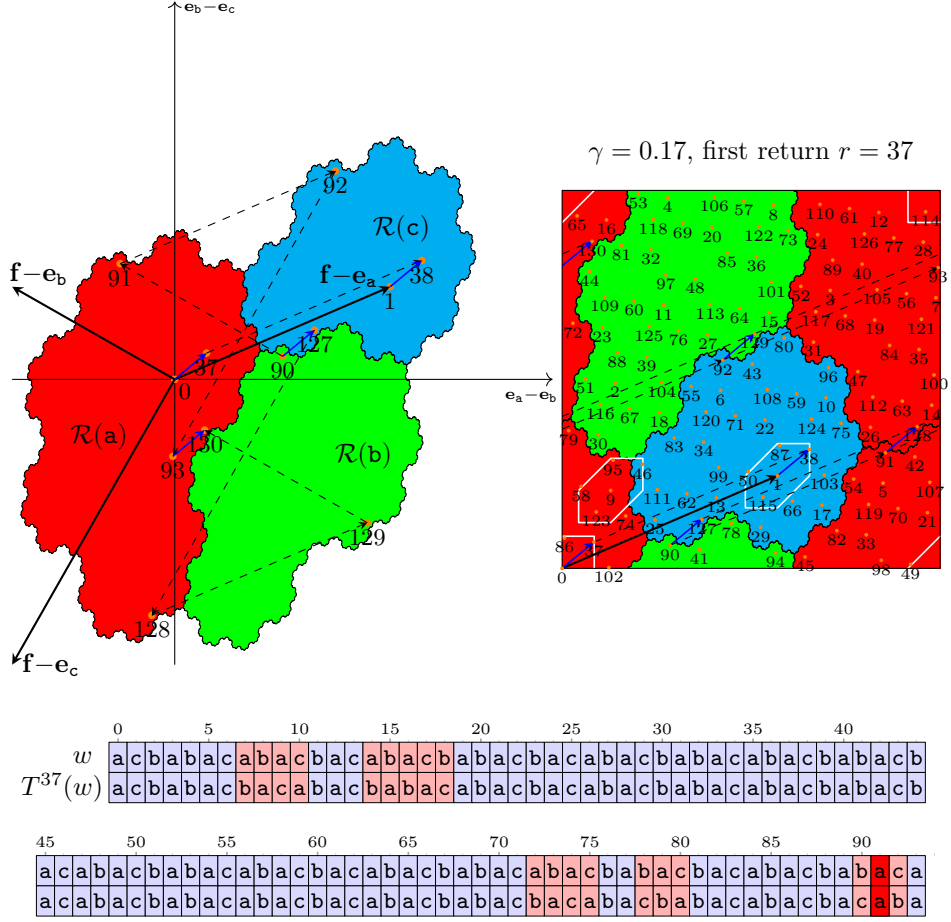}
    \caption{Rauzy fractals $\mathcal{R}(i)$, $i \in \{\mathtt{a}, \mathtt{b}, \mathtt{c}\}$, for the Cassaigne--Selmer shift with $\bsigma = (\gamma_1, \gamma_2)^{\infty}$ providing a domain exchange with translation vectors $\mathbf{f}{-}\mathbf{e}_i$, $\bff \approx (0.43016,0.32472,0.24512)$, their representations on the torus $\mathbf{1}^\perp / (\mathbb{Z}^3 \cap \mathbf{1}^\perp)$, and the length $s = 94$ prefix of the fixed point~$w$ of $\gamma_1 \circ \gamma_2$. \\
      For $r=37$, the action of the shift by~$T^r$ is illustrated in the three settings.
      The set $\cJ_r(w) \cap [0,94)$ is represented by indices~$i$ with red background, with a mismatch $w_i \ne w_{i+r}$ for all $i$ except $i = 91$; in the latter case, we have $\bl(w_{[i,i+r)}) = (16,11,10) \ne (16,12,9) = \round{r\bff}$.
      When $w_i \ne w_{i+r}$, the points $\phi \circ T^i(w)$ and $\phi \circ T^{i+r}(w)$ are not in the same subtile.
      The condition $\bl(w_{[i,i+r)}) \ne \round{r\bff}$ is equivalent to $\phi \circ T^{i+r}(w) \ne \phi \circ T^i(w){+}r\bff{-}\round{r\bff}$ but we may have $w_i = w_{i+r}$. \\
      Balls of radius $\gamma = 0.17$ (in the 1-norm of $\RR^{3}$), which contain the translation $r\bff {-} \round{r\bff}$ and are used in the proof of Lemma~\ref{lem:density0}, are also represented (in white).} \label{f:R}
\end{figure}

\begin{proof}
For the sake of readability, we omit the index $\bsigma$ from notation such as~$\cC_{\bsigma}$.

First note that (\ref{i:d04}) holds for all sufficiently large $s$ because the primitivity of $\bsigma$ implies that $X$ is uniformly recurrent, thus the set $\bigcup_{v\in\cL\cap\cA^{r+1}} R_X(v)$ is finite and occurs in each sufficiently long element of~$\cL$.

The tiling property of~$\cC$ with compact sets implies that the vector~$\bff$ has rationally independent coordinates (because otherwise $\mathcal{R}$ would be in a finite union of proper affine subspaces of~$\bone^\perp$).
Moreover, we have $\lambda(\partial\cR(a)) = 0$ for all $a \in \cA$, where $\lambda$ is the Lebesgue measure on~$\bone^\perp$. 
By \cite[Lemma~8.3]{Berthe19Geometrydynamics}, where recurrence and algebraic irreducibility can be replaced by the rational independence of the frequency vector by \cite[Lemma~5.11]{Berthe23Multidimensionalcontinued} and balancedness is equivalent to the compactness of the Rauzy fractal by \cite[Lemma~4.1]{Berthe19Geometrydynamics}, there exists a map $\phi: X \to \bone^\perp$ satisfying, for all $w_0w_1\cdots \in X$,
\[
\phi(w_0w_1\cdots) \in \cR(w_0)\ \mbox{and}\ \phi \circ T(w_0w_1\cdots) = \phi(w_1w_2\cdots) = \phi(w_0w_1\cdots) {+} \bff {-} \be_{w_0}.
\]

For $\Lambda := \ZZ^{\cA}\cap\bone^\perp$, let $\pi: \bone^\perp \to \bone^\perp/\Lambda$ be the canonical projection $\bx \to \bx + \Lambda$.
Define the translation
\[
\tilde{T}:\ \bone^\perp/\Lambda \to \bone^\perp/\Lambda, \quad \bx + \Lambda \mapsto \bx + \bff - \be_a + \Lambda,
\]
for some $a \in \cA$.
(The definition does not depend on~$a$ because $\be_a-\be_b \in \Lambda$ for all $a,b \in \cA$.)
Then we have the commutative diagram
\[
\begin{tikzcd}
X \arrow [swap,d,"\pi\circ\phi"] \arrow [r,"T"] & X \arrow[d,"\pi\circ\phi"] \\
\bone^\perp/\Lambda \arrow[r,"\tilde{T}"] & \bone^\perp/\Lambda
\end{tikzcd}
\]
By the proof of \cite[Proposition~8.4]{Berthe19Geometrydynamics}, we have the measure theoretic conjugacy 
\[
(X,T,\mu) \stackrel{\pi\circ\phi}{\to} (\bone^\perp/\Lambda,\tilde{T},\tilde{\lambda}),
\] 
where $\mu$ is the unique invariant probability measure of $(X,T)$ (see \cite[Theorem~5.7]{Berthe14substitutivedynamical}) and $\tilde{\lambda}$ the Haar measure on $\bone^\perp/\Lambda$.
The tiling property of~$\cC$ implies
\[
\tilde{\lambda}(\pi(\cR(a))) = \lambda(\cR(a)) = f_a.
\]

For $\gamma \ge 0$, $a \in \cA$, let
\[
\partial_{\gamma}\cR(a) := \{\bx \in \bone^\perp : \norm{\bx-\by}_1 < \gamma\ \mbox{for some}\ \by\in\partial\cR(a)\}.
\]
Since $\lambda(\partial\cR(a)) = 0$, we have 
\[
\lim_{\gamma\to0} \lambda(\partial_\gamma\cR(a)) = 0
\]
for all $a \in \cA$.
Therefore, we can choose $\gamma > 0$ such that
\[
\lambda(\partial_{2\gamma}\cR(a)) \le \varepsilon\, f_a/3
\]
for all $a \in \cA$.
By Kronecker's theorem, there exists an integer $r \ge 1$ such that
\begin{equation} \label{e:ruz}
\norm{r\bff-\round{r\bff}}_1 < \gamma.
\end{equation}
Note that $\norm{r\bff}_1 = r$ and $\norm{\round{r\bff}}_1 \in \mathbb{Z}$ implies that $\norm{\round{r\bff}}_1 = r$ (for $\gamma \le 1$).
For all $w = w_0w_1\cdots \in X$, 
\[
i \in \cJ_r(w) \quad \mbox{implies that} \quad \phi \circ T^i(w) \in \partial_\gamma \cR(w_i).
\]
Indeed, we have $\phi \circ T^i(w) \in \cR(w_i)$ and $\phi \circ T^{i+r}(w) = \phi \circ T^i(w) + r\bff - \bl(w_{[i,i+r)})$ by the definition of~$\phi$.
If $\phi \circ T^i(w) \notin \partial_\gamma \cR(w_i)$, then $\phi \circ T^i(w) + r\bff - \round{r\bff}$ is in the interior of $\cR(w_i)$.
Since $\cC$ forms a tiling of $\bone^\perp$, this implies that $\bl(w_{[i,i+r)}) = \round{r\bff}$ and $w_{i+r} = w_i$, i.e., $i \notin \cJ_r(w)$.

Moreover, we have $\tilde{T}^r(\bx) \in B_\gamma(\bx)$ for all $\bx \in \bone^\perp/\Lambda$, where $B_\gamma(\bx)$ is the open ball in $\bone^\perp/\Lambda$ with radius~$\gamma$ and center~$\bx$ associated with the distance induced by the norm~$\norm{\cdot}_1$ on~$\bone^\perp$.
By the ergodic theorem, we have, for almost all $\bx \in \bone^\perp/\Lambda$, 
\begin{equation}\label{eq:time average}
\begin{aligned}
\lim_{s\to\infty} \frac{1}{s}\, |\{0\le i < s \,:\, \tilde{T}^i(\bx) \in \pi(\cR(a))\}| & = \tilde{\lambda}(\pi(\cR(a))) = f_a, \\
\lim_{s\to\infty} \frac{1}{s}\, |\{0\le i < s \,:\, \tilde{T}^i(\bx) \in \pi(\partial_{2\gamma}\cR(a))\}| & = \tilde{\lambda}(\pi(\partial_{2\gamma}\cR(a))) \le \lambda(\partial_{2\gamma}\cR(a)).
\end{aligned}
\end{equation}
Since $\bone^\perp/\Lambda$ is compact, it has a cover by finitely many open balls $B_\gamma(\bx)$ with $\bx$ satisfying \eqref{eq:time average}.
Let $\varepsilon'< \min(\varepsilon/2, 1/72)$.
Choose $s$ large enough such that (\ref{i:d02}) and (\ref{i:d04}) hold and
\[
\begin{aligned}
|\{0\le i < s/2 \,:\, \tilde{T}^i(\bx) \in \pi(\cR(a))\}| & \ge (1/2-\varepsilon') f_a\, s, \\
(1/8-\varepsilon') f_a\, s \leq |\{0\le i < s/8 \,:\, \tilde{T}^i(\bx) \in \pi(\cR(a))\}| &\le (1/8+\varepsilon') f_a\, s, \\
|\{0\le i < s \,:\, \tilde{T}^i(\bx) \in \pi(\partial_{2\gamma}\cR(a))\}| & \le \varepsilon'\, f_a\, s,
\end{aligned}
\]
for all $\bx$ in this cover and $a \in \cA$.

Let $\pi \circ \phi(w) \in B_\gamma(\bx)$ with $\bx$ from the finite cover.
For $a \in \cA$, we have
\[
\begin{gathered}
|\{i \in \cJ_r(w) \cap [0,s) : w_i = a\}| \le |\{0 \le i < s \,:\, \tilde{T}^i(\bx) \in \pi(\partial_{2\gamma}\cR(a))\}| \le \varepsilon\, f_a\, s/2, \\
\begin{aligned}
|w_{[1,s/8)}|_a  & \le |\{1 \le i < s/8 \,:\, \phi \circ T^i(w) \in \cR(a) \}| \\
& \le |\{1 \le i < s/8 \,:\, \tilde{T}^i(\bx) \in \pi(\cR(a) \cup \partial_{2\gamma}\cR(a)) \}| \le f_a\, s/7, \\
|w_{[1,s/8)}|_a & \ge |\{1 \le i < s/8 \,:\, \phi \circ T^i(w) \in \cR(a) \setminus \partial\cR(a) \}| \\
& \ge |\{1 \le i < s/8 \,:\, \tilde{T}^i(\bx) \in \pi(\cR(a) \setminus \partial_{2\gamma}\cR(a))\}| \ge f_a\, s/9,  \\
\end{aligned}
\end{gathered}
\]
and $|w_{[1,s/2)}|_a \ge f_a\, s/3$, which proves~(\ref{i:d01}) and~(\ref{i:d03}). 
\end{proof}

Recall that the Definition~\ref{def:echoing} of echoing requires a finite number of intervals in condition~\eqref{i:243}.
The recurrence of the directive sequence guarantees that the length $r+s$ prefixe of $w$ at level $0$ appears at infinitely many levels; applying the substitutions to this fixed prefix ensures the condition~\eqref{i:243}.

\begin{lemma}\label{lem:Echoing}
Let $\bsigma$ be a primitive and recurrent sequence of unimodular substitutions over~$\cA$, and assume that $\cC_{\bsigma}$ forms a tiling of $\bone^\perp$ with compact sets.
Then each $w \in X_{\bsigma}$ is echoing.
\end{lemma}

\begin{proof}\label{proof:Echoing}
Let $\bsigma = (\sigma_{n})_{n\in\NN}$ be as in Lemma~\ref{lem:density0}.
Let $h \ge 1$ be such that $|\sigma_{[0,h)}(a)| \ge s{+}r$ for all $a \in \cA$, and by the recurrence, let $(n_k)_{k\in\NN}$ be an increasing sequence of natural numbers such that $\sigma_{[n_k,n_k+h)} = \sigma_{[0,h)}$ for all $k \in \NN$. 

Let $w \in X_{\bsigma}$.
Then there exist $w_0^{(k)}w_1^{(k)}\cdots \in X_{T^{n_k}\bsigma}$, $0 \le j_k < |\sigma_{[0,n_k)}(w_0^{(k)})|$, such that 
\[
w = T^{j_k} \sigma_{[0,n_k)}(w_0^{(k)}w_1^{(k)}\cdots).
\]
Let 
\begin{equation} \label{e:rk}
r_{k} := \norm{M_{\sigma_{[0,n_k)}} \round{r\bff}}_1
\end{equation}
(where we recall that $\round{r\bff}$ is the nearest integer vector to~$r\bff$), 
\begin{equation} \label{e:sk}
s_k := |\sigma_{[0,n_k)}(w^{(k)}_{[0,s)})| - j_k, 
\end{equation}
and 
\[
\cI_k := \bigcup_{i\in\cJ_{r}(w^{(k)})\cap[0,s)} \big[|\sigma_{[0,n_k)}(w^{(k)}_{[0,i)})|{-}j_k, |\sigma_{[0,n_k)}(w^{(k)}_{[0,i+1)})|{-}j_k\big) \cap \NN,
\]
with $\cJ_r$ defined in \eqref{e:Jr}.
If $i \notin \cJ_{r}(w^{(k)}) \cap [0,s)$, then $w^{(k)}_i = w^{(k)}_{i+r}$ and
\[
r_k = \norm{M_{\sigma_{[0,n_k)}} \bl(w^{(k)}_{[i,i+r)})}_1 = \norm{\bl(\sigma_{[0,n_k)}(w^{(k)}_{[i,i+r)}))}_1 = |\sigma_{[0,n_k)}(w^{(k)}_{[i,i+r)})|.
\]
This implies that
\[ 
    \{0 \le i<s_k : w_i \ne w_{i+r_{k}}\} \subset \cI_k,
\]
and $\cI_k$ is a union of at most $s$ intervals of integers.

Since $|\sigma_{[0,h)}(a)| \ge s{+}r$ for all $a \in \cA$, $w_{[0,s+r)}^{(k)}$ is a factor of $\sigma_{[0,h)}(ab)$ for some $a,b \in \cA$.
Consider first the case that $w_{[0,s+r)}^{(k)}$ is a factor of $\sigma_{[0,h)}(a)$.
Then $w_{[0,s+r)}^{(k)}$ is the prefix of some $\tilde{w} = \tilde{w}_0\tilde{w}_1\cdots \in X_{\bsigma}$.
Then we have 
\[
\begin{gathered}
r_{k} \le r \max_{a\in\cA} |\sigma_{[0,n_k)}(a)| \le \varepsilon\, s \min_{a\in\cA} f_a \max_{a\in\cA} |\sigma_{[0,n_k)}(a)| \le \varepsilon\,s \sum_{a\in\cA} f_a |\sigma_{[0,n_k)}(a)|, \\
|\cI_k| \le \sum_{a \in A} |\{ i \in \cJ_{r}\cap[0,s) : \tilde{w}_i = a \}|\ |\sigma_{[0,n_k)}(a)| \le \varepsilon\,s \sum_{a\in\cA} f_a\, |\sigma_{[0,n_k)}(a)|, \\
s \sum_{a\in\cA} f_a |\sigma_{[0,n_k)}(a)| \le 3 \sum_{a\in\cA} |\tilde{w}_{[1,s/2)}|_a \  |\sigma_{[0,n_k)}(a)| < 3 s_k.
\end{gathered}
\]

Consider now the case that $w_{[0,s')}^{(k)}$ is a suffix of $\sigma_{[0,h)}(a)$ for some $s' < s+r$ and $w_{[s',s+r)}^{(k)}$ is a prefix of $\sigma_{[0,h)}(b)$.
Then $w_{[0,s')}^{(k)}$ is the prefix of some $w' \in X_{\bsigma}$ and $w_{[s',s+r)}^{(k)}$ is the prefix of some $\tilde{w} \in X_{\bsigma}$.
We have
\begin{align*}
    |\cI_k| &\le \sum_{a \in A} \Big( \big| \{ i \in \cJ_{r}\cap[0,s) : \tilde{w}_i = a \} \big|+ r + \big| \{ i \in \cJ_{r}\cap[0,s) : w'_i = a \} \big|\Big) |\sigma_{[0,n_k)}(a)| \\ 
    &\le 3 \varepsilon\,s \sum_{a\in\cA} f_a\, |\sigma_{[0,n_k)}(a)|.
\end{align*} 
If $s/2 \le s'$, then $w'_{[1,s/2)}$ is a factor of $w_{[0,s')}^{(k)}$, therefore we have
\[
    s \sum_{a\in\cA} f_a |\sigma_{[0,n_k)}(a)| \le 3 \sum_{a\in\cA} | w'_{[1,s/2)}|_a \, |\sigma_{[0,n_k)}(a)| < 3 s_k.
\]
If $s/2 > s'$, then $\tilde{w}_{[1,s/2)}$ is a factor of $w_{[s',s+r)}^{(k)}$ and we can apply the above argument to get the same bound.
Therefore, 
\[ 
r_{k} + |\cI_k| < 9 \varepsilon s_k.
\]
This proves that $w$ is echoing.
\end{proof}

\section{Strong echoing of dendric $\cS$-adic words}\label{sec:Strong echoing of dendric S-adic words}
Let $\bsigma$ be a primitive directive sequence of unimodular substitutions over a finite alphabet~$\cA$ and $X_{\bsigma}$ be the $\cS$-adic shift generated by~$\bsigma$.
By \cite{Berthe14substitutivedynamical}, primitivity gives that $(X_{\bsigma},T)$ is a minimal shift (i.e., all its elements are uniformly recurrent).
Recall that, for a factor $u$ of a word in~$X_{\bsigma}$, $R_{X_{\bsigma}}(u)$ is the set of return words, and $X_{\bsigma}$ is dendric if each set $R_{X_{\bsigma}}(u)$ forms a basis of the free group~$F_{\cA}$.

In order to show that the non-vanishing condition (\ref{eq:non vanishing}) of Definition~\ref{def:strongechoing} holds, we define a map that acts on $\cA^*$ by $v_1 \cdots v_n \mapsto \sum_{i=1}^{n} v_i \beta^{n-i}$.
Since we want to use the free group properties of dendric words, we have to extend this map to~$F_{\cA}$, and we define the map as composition of a homomorphism on~$F_{\cA}$ and an evaluation map that produces $\beta$-representations. 

Consider the group $(\ZZ^{\cA}[x,x^{-1}],*)$ of Laurent polynomials with coefficients in~$\ZZ^{\cA}$ with the group operation~$*$ defined by
\[
f(x) * g(x) := f(x)\, x^{\ell(g)} + g(x),\quad \mbox{where}\ \ell(g) := \langle g(1),\bone \rangle;
\]
recall that $\bone = (1)_{a\in\cA}$, thus $\langle g(1),\bone \rangle = \sum_{a\in\cA} z_a$ for $g(1) = (z_a)_{a\in\cA} \in \ZZ^{\cA}$ and note that $\ell(f*g) = \ell(f) + \ell(g)$.
The inverse element of $f(x)$ is $g(x) = - f(x) x^{-\ell(f)}$ and satisfies $\ell(g) = - \ell(f)$.
Define the group homomorphism
\[
P:\ (F_{\cA},\cdot) \to (\ZZ^{\cA}[x,x^{-1}],*) \quad \mbox{by}\ a \mapsto \be_a \ \mbox{for}\ a \in \cA,
\]
where $\cdot$ denotes the free group operation.
Then $P(v_1\cdots v_{n})(x) = \sum_{i=1}^{n} \be_{v_i} x^{n-i}$ for $v_1\cdots v_{n} \in \cA^{n}$ and $P(a^{-1})(x) = -\be_a x^{-1}$ for $a \in \cA$.
In particular, we have $\ell(P(v)) = |v|$ and $\ell(P(v^{-1})) = -|v|$ for $v \in \cA^*$.
For $\cA \subset \CC$ and $\beta \in \CC \setminus \{0\}$, define the $\beta$-evaluation map
\[
\mathrm{ev}_\beta:\ \ZZ^{\cA}[x,x^{-1}] \to \CC, \quad f \mapsto \langle f(\beta), (a)_{a\in\cA}\rangle, 
\]
thus $\mathrm{ev}_\beta(\sum_i \sum_{a\in\cA} z_{i,a} \be_a x^i) = \sum_i \sum_{a\in\cA} z_{i,a} a \beta^i$ and, in particular, 
\[
\mathrm{ev}_\beta(P(v_1\cdots v_{n})) = \sum_{i=1}^{n} v_i \beta^{n-i} \quad \mbox{for}\ v_1\cdots v_{n} \in \cA^{n}.
\]

\begin{lemma} \label{lem:returnwords}
Let $\cA \subset \CC$, $L \subset F_{\cA}$, $\beta \in \CC$ not a root of unity, and  $u \in F_{\cA}$ such that $\ell(P(u)) \neq 0$.
Assume that $\mathrm{ev}_\beta(P(uv)) = \mathrm{ev}_\beta(P(vu))$ for all $v \in L$.
Then $\mathrm{ev}_\beta(P(w)) = \mathrm{ev}_\beta(P(w'))$ for all $w,w' \in \langle uL\rangle$ such that $\ell(P(w)) = \ell(P(w'))$, in particular $|\langle uL\rangle \cap \cA| \le 1$.
\end{lemma}

\begin{proof}
First note that $\mathrm{ev}_\beta(f*g*h) = \mathrm{ev}_\beta(f*g'*h)$ if $\mathrm{ev}_\beta(g) = \mathrm{ev}_\beta(g')$ and $\ell(g) = \ell(g')$.
Therefore, we have $\mathrm{ev}_\beta(P(w)) = \mathrm{ev}_\beta(P(w'))$ for all pairs $w,w'$ that are obtained by replacing $uv$ by~$vu$, $v \in L$.
In particular, we have $\mathrm{ev}_\beta(P(uw)) = \mathrm{ev}_\beta(P(wu))$ for all $w \in \langle uL\rangle$, the subgroup of $F_{\cA}$ generated by $\{uv : v \in L\}$, i.e., 
\[
\mathrm{ev}_\beta(P(u))\, \beta^{\ell(P(w))} + \mathrm{ev}_\beta(P(w)) = \mathrm{ev}_\beta(P(w))\,  \beta^{\ell(P(u))} + \mathrm{ev}_\beta(P(u)),
\]
thus
\[
\mathrm{ev}_\beta(P(w)) = \frac{\beta^{\ell(P(w))}-1}{\beta^{\ell(P(u))}-1}\, \mathrm{ev}_\beta(P(u)).
\]
This shows that $\mathrm{ev}_\beta(P(w)) = \mathrm{ev}_\beta(P(w'))$ for all $w, w' \in \langle uL\rangle$ with $\ell(P(w)) = \ell(P(w'))$.
For $a \in \cA$, we have $\mathrm{ev}_\beta(P(a)) = a$ and $\ell(P(a)) = 1$, thus $a, b \in \langle uL\rangle \cap \cA$ implies $a = b$.
\end{proof}

Let $p$ be a natural number.
A finite word $v$ is called \emph{$p$-periodic} if $v_i = v_{i+p}$ for all $0 \le i < |v|{-}p$.
The following lemma shows that dendric $\cS$-adic shifts have periodic finite words.

\begin{lemma}\label{lem:expanding gap on return words}
Let $\bsigma$ be a primitive sequence of unimodular substitutions such that $X_{\bsigma}$ is dendric.
Let $n,r \ge 1$, $v = v_{[0,r)} \in \cA^r$, and $w \in X_{\bsigma}$ be such that $\sigma_{[0,n)}(v\, v_0)$ is a factor of~$w$.
Then there exists a factor $u$ of $\sigma_{[0,n)}(v\,v_0)$ with $|u| = |\sigma_{[0,n)}(v)|$ such that, for every $j \in \NN$, $w_{[j,j+|u|)} = u$ implies 
\[
w_{i} = w_{i+|u|} \quad \text{for all integers } i \in \bigg[j, j+ \bigg\lfloor \frac{|\sigma_{[0,n)}(v_0)|}{|\cA|} \bigg\rfloor \bigg)\cdot
\]
If moreover $\tilde{w} = \tilde{w}_0 \tilde{w}_1 \cdots \in \cA^\NN$ is such that $\sigma_{[0,n)}(\tilde{w}_{[0,s/8)})$ is a factor of~$w$ and $\tilde{w}_{[0,s/8)}$ contains a set of return words of~$v\,v_0$ with abelianizations spanning~$\RR^{\cA}$, then all elements of $R_{X_{\bsigma}}(u)$ are factors of $\sigma_{[0,n)}(\tilde{w}_{[0,s/8)})$.
\end{lemma}

\begin{proof}
Since $X_{\bsigma}$ is dendric, the number of factors of length~$k$ is $(|\cA|{-}1)k{+}1$ for all $k\in\NN$.
Therefore, for all $t \in \NN$, at most $(|\cA|{-}1)t$ factors of length $k$ have more than one right extension of length~$t$.
Suppose that all factors of $\sigma_{[0,n)}(v\,v_0)$ of length $\tilde{r} = |\sigma_{[0,n)}(v)|$ have more than one right extension of length $t = \lfloor|\sigma_{[0,n)}(v_0)|/|\cA|\rfloor$.
Then one of these factors appears twice in $\sigma_{[0,n)}(v\,v_0)$, at positions bounded by $(|\cA|{-}1)t$.
Since $|\cA|t \le |\sigma_{[0,n)}(v_0)|$, there exists a length $\tilde{r}$ factor $u$ of $\sigma_{[0,n)}(v)$ which is $p$-periodic word for some $1 \le p \le (|\cA|{-}1)t$.
Since $\sigma_{[0,n)}(v\,v_0)$ is $\tilde{r}$-periodic, we obtain that the length~$p$ right extension~$u'$ of~$u$ is also $p$-periodic and $\tilde{r}$-periodic.
By Fine and Wilf's theorem on $u'$, w.l.o.g., $p$ divides~$\tilde{r}$.
In particular, we have $\bl(\sigma_{[0,n)}(v)) \in (\tilde{r}/p)\, \ZZ^{\cA}$.
By the unimodularity of $\sigma_{[0,n)}$, we obtain that $\bl(v) \in (\tilde{r}/p)\, \ZZ^{\cA}$, thus $p\, |\{0 \le i < r : v_i = v_0\}| \in \tilde{r}\, \ZZ$, giving the contradiction 
\[
\tilde{r} \le p\,|\{0 \le i < r : v_i = v_0\}| < |\sigma_{[0,n)}(v_0)|\, |\{0 \le i < r : v_i = v_0\}| \le |\sigma_{[0,n)}(v)| = \tilde{r}.
\]
This implies that $\sigma_{[0,n)}(v\,v_0)$ has a factor $u \in \cA^{\tilde{r}}$ occurring before position $(|\cA|{-}1)t$ with a unique extension of length~$t$.
Since $|\cA|t \le |\sigma_{[0,n)}(v_0)|$, this extension is a prefix of~$u$.
Hence, $w_{[i,i+\tilde{r})} = u$ implies $w_{[i,i+t)} = w_{[i+\tilde{r},i+\tilde{r}+t)}$ for all $w \in X_{\bsigma}$, $i \in \NN$.

If $\tilde{w}_{[0,s/8)}$ contains return words of $v$ with abelianizations spanning~$\RR^{\cA}$, then the $\sigma_{[0,n)}$-images of these return words also span~$\RR^{\cA}$.
Therefore, the abelianizations of  return words of~$u$ that are factors of $\sigma_{[0,n)}(\tilde{w}_{[0,s/8)})$ also span~$\RR^{\cA}$.
By dendricity, this implies that each element of $R_{X_{\bsigma}}(u)$ is a factor of $\sigma_{[0,n)}(\tilde{w}_{[0,s/8)})$.
\end{proof}

\begin{lemma}\label{lem:strong echoing and return word}
Let $w = w_0 w_1 \cdots \in \cA^{\mathbb{N}}$ be an echoing word with sequences $(r_{k})_{k\in\NN}$, $(s_k)_{k\in\NN}$.
Assume that, for every $k \in \NN$, $w$~has a factor $u = u_{[0,r_k)} \in \cA^{r_{k}}$ such that
\begin{enumerate}[\upshape (i)]
\item \label{i:531}
every occurrence of~$u$ in~$w_{[0,s_{k}/2+r_{k})}$ is followed by a prefix~$u_{[0,l_k)}$ of~$u$ with $\lim_{k \to \infty} l_{k}/\log(s_{k})=\infty$, 
\item \label{i:532}
and $w_{[0,s_{k}/2)}$ contains a set of return words of~$u$ such that the generated subgroup of~$F_\cA$ contains at least two different elements of~$\cA$.
\end{enumerate}
Then $w$ is strong echoing.
\end{lemma}

\begin{proof}
Let $\beta \in \overline{\QQ}$ that is not a root of unity.
By Lemma~\ref{lem:returnwords} and Assumption~(\ref{i:532}), there exists a return word~$uv$ of~$u$ in $w_{[0,s_{k}/2)}$ such that $\mathrm{ev}_\beta(P(uv)) \neq \mathrm{ev}_\beta(P(vu))$.
Here, we have $v \in F_{\cA}$, i.e., $uv$ can be shorter than~$u$. 
However, we have $|uv| > l_k$, because $|uv| \le l_k$, the return word property of~$uv$, and Assumption~(\ref{i:531}) would imply that the suffix $uvu$ of length $|uv|$ was a prefix of~$u$, thus $vu = uv$, contradicting that $\mathrm{ev}_\beta(P(uv)) \neq \mathrm{ev}_\beta(P(vu))$.

Since $uv$ occurs as a return word of~$u$ in $w_{[0,s_{k}/2)}$, we can choose $p_k, q_k$ with $l_k \le p_k < q_k < s_k/2$ such that
\[ 
w_{[p_{k}-l_{k},q_{k}+r_k]} = u v u.
\]
Then we have
\[
  \mathrm{ev}_\beta(P(w_{[p_{k}-l_{k},q_{k}]}))=
  \mathrm{ev}_\beta(P(uv)) \neq
  \mathrm{ev}_\beta(P(vu)) =
  \mathrm{ev}_\beta(P(w_{[p_{k}+r_{k}-l_{k},q_{k}+r_{k}]})).
\]
Since they are both prefixes of~$u$, we have $w_{[p_{k}-l_{k},p_{k})} = w_{[p_{k}+r_{k}-l_{k},p_{k}+r_{k})}$ by Assumption~\eqref{i:531}. 
This implies that
\[
\mathrm{ev}_\beta(P(w_{[p_{k},q_{k}]})) \neq \mathrm{ev}_\beta(P(w_{[p_{k}+r_{k},q_{k}+r_{k}]}))
\]
Multiplying by $\beta^{-q_{k}}$ gives the condition~\eqref{eq:non vanishing} of strong echoing.
We also have $w_{[q_{k}+1,q_{k}+l_k]} = w_{[q_{k}+r_{k}+1,q_{k}+r_{k}+l_k]}$ (by Assumption~(\ref{i:531})), which gives the condition~\eqref{eq:expanding gap} of strong echoing.
\end{proof}

We can now prove the second statement of Theorem~\ref{thm:S-adic echoing}.
\begin{proposition}\label{prop:strong echoing of dendric S-adic words}
    Let $\bsigma$ be a primitive and recurrent sequence of unimodular substitutions over~$\cA$, and assume that $\cC_{\bsigma}$ forms a tiling of $\bone^\perp$ with compact sets.
    If $X_{\bsigma}$ is dendric, then each $w \in X_{\bsigma}$ is strong echoing.
\end{proposition}

\begin{proof}\label{proof:strong echoing of dendric S-adic words}
    Let $\varepsilon \leq 1/9$.
    Let $r$, $s$ be as in Lemma~\ref{lem:density0}, and let $n_k$, $w^{(k)}, w',\tilde{w}$, $s'$ be as in the proof of Lemma~\ref{lem:Echoing}.
    Let $a \in \cA$ be a letter such that $|\sigma_{[0,n_k)}(a)| \ge |\sigma_{[0,n_k)}(b)|$ for all $b \in \cA$.
    By Lemma~\ref{lem:density0}~(\ref{i:d01})--(\ref{i:d02}), we have 
    \[
        |\{i  \in \cJ_{r}(\tilde{w}) \cap [0, s/8) : a = w_i\}|
        \leq \varepsilon f_{a} \, s \leq f_{a} s/9 < |\tilde{w}_{[1,s/8)}|_a.
    \]
    Thus there exists a factor $v v_{0} \in \cA^{r+1}$ of $\tilde{w}_{[1, s/8+r)}$ such that $v_0 = a$.
    By Lemma~\ref{lem:density0} (\ref{i:d04}), we can assume that $v v_{0} v'$~is a factor of $\tilde{w}_{[1, s/8)}$ for all $v' \in R_{X_{\bsigma}}(v v_{0})$.
    By Lemma~\ref{lem:expanding gap on return words}, there exists a factor $u$ of $\sigma_{[0,n_k)}(v v_{0})$ with $|u| = |\sigma_{[0,n_k)}(v)|(=r_k)$ such that, for all index $j$ of an occurrence of $u$ in $w$,
    \[
    w_{i} = w_{i+|u|} \quad \text{for all integers } i \in \bigg[j, j+ \bigg\lfloor \frac{|\sigma_{[0,n_k)}(v_0)|}{|\cA|} \bigg\rfloor \bigg),
    \]
    and all the return words of $u$ are factors of $\sigma_{[0,n_k)}(\tilde{w}_{[1,s/8)})$.
    When $k\to\infty$, we have $|\sigma_{[0,n_k)}(v_0)| / (|\cA| \log s_k) \to \infty$  since $s_k \le |\sigma_{[0,n_k)}(v_0)|\, s$ and $\bsigma$ is primitive.
    We have $s_{k}/2 \geq (1/2) \, \sum_{a \in \cA} |\sigma_{[0,n_k)}(a)| \times |\{ 1 \leq i <  s : w^{(k)}_{i} = a \}|$ and
    \[
        |w^{(k)}_{[1,s)}|_a \geq |\tilde{w}_{[1,s/2)}|_a \geq \frac{f_{a}\, s}{3}
        \geq \frac{2f_{a}\, s}{7} \geq 2\, |\tilde{w}_{[1,s/8)}|_a.
    \]
    Thus, we have $s_{k}/2 > |\sigma_{[0,n_k)}(\tilde{w}_{[1,s/8)})|$.
    Therefore, all the return words of $u$ are factors of $w_{[0,s_k/2)}$.
    Since $X_{\bsigma}$ is dendric, they generate the free group~$F_{\cA}$.
    Then, by Lemma~\ref{lem:strong echoing and return word}, $w$ is strong echoing.
\end{proof}

\bibliographystyle{amsalpha} 
\bibliography{echoing}

\end{document}